\documentclass[10pt]{article}
\usepackage[a4paper,centering,vscale=0.75,hscale=0.7]{geometry}
\usepackage[T1]{fontenc}
\usepackage{mathtools,
  amsthm,
  amssymb,
  mathrsfs,
  bbm}

\newtheorem{prop}{Proposition}[section]
\newtheorem{thm}[prop]{Theorem}
\newtheorem{lemma}[prop]{Lemma}
\newtheorem{cor}[prop]{Corollary}
\theoremstyle{definition}

\newtheorem{hyp}{Hypothesis}
\theoremstyle{remark}
\newtheorem{rmk}[prop]{Remark}

\newcommand{\ind}[1]{\mathbbm{1}_{#1}}

\newcommand{\dom}{\mathsf{D}}
\newcommand{\E}{\mathbb E}
\newcommand{\cF}{\mathscr{F}}
\newcommand{\cL}{\mathscr{L}}
\renewcommand{\P}{\mathbb{P}}
\newcommand{\enne}{\mathbb{N}}
\newcommand{\erre}{\mathbb{R}}

\newcommand{\embed}{\hookrightarrow}
\newcommand{\longto}{\longrightarrow}
\DeclareMathOperator{\Tr}{Tr}

\DeclarePairedDelimiter{\abs}{\lvert}{\rvert}
\DeclarePairedDelimiter{\norm}{\lVert}{\rVert}
\DeclarePairedDelimiterX\ip[2]{\langle}{\rangle}{#1,#2}

\DeclarePairedDelimiterX\oo[2]{]\!]}{[\![}{#1,#2}
\DeclarePairedDelimiterX\cc[2]{[\![}{]\!]}{#1,#2}
\DeclarePairedDelimiterX\co[2]{[\![}{[\![}{#1,#2}
\DeclarePairedDelimiterX\oc[2]{]\!]}{]\!]}{#1,#2}

\title{Positivity of mild solution to stochastic evolution equations with
  an application to forward rates}

\author{Carlo Marinelli\thanks{Department of Mathematics, University
    College London, Gower Street, London WC1E 6BT, UK. URL:
    \texttt{goo.gl/4GKJP}}}

\date{\normalsize April 1, 2019}

\begin{document}

\maketitle

\begin{abstract}
  We prove a maximum principle for mild solutions to stochastic
  evolution equations with (locally) Lipschitz coefficients and Wiener
  noise on weighted $L^2$ spaces. As an application, we provide
  sufficient conditions for the positivity of forward rates in the
  Heath-Jarrow-Morton model, considering the associated Musiela SPDE
  on a homogeneous weighted Sobolev space.
\end{abstract}


\section{Introduction}
Consider stochastic evolution equations of the type
\begin{equation}
  \label{eq:0}
  du(t) + Au(t)\,dt = f(u(t))\,dt + B(u(t))\,dW(t),
  \qquad u(0)=u_0,
\end{equation}
on a Hilbert function space $H$, where $A$ is a linear maximal
monotone operator on $H$, $W$ is a cylindrical Wiener process, and the
coefficients $f$ and $B$, which can be random and time-dependent,
satisfy suitable measurability and Lipschitz continuity
conditions. Precise assumptions on the data of the problem are given
in {\S}\ref{sec:main} below. Our goal is to establish conditions on the
coefficients $A$, $f$, and $B$ implying that the ``flow'' associated
to \eqref{eq:0} is positivity preserving, i.e. such that its solution
is positive at all times provided the initial datum is positive.

Maximum principles for stochastic PDEs have a long history, and some
references to the earlier literature can be found in
\cite{Kry:MP-SPDE}. Most recent contributions seem to consider
equations that admits a variational formulation (in the sense of
\cite{KR-spde,Pard}). As examples of such results, let us mention
\cite{Kry:MP-SPDE, Kry:shortIto}, where large classes of second-order
stochastic PDEs are considered, as well as \cite{DaGy} for the case of
equations driven by certain classes of discontinuous noise. Further
classes of second-order parabolic stochastic PDEs that can be treated
by techniques reminiscent of the variational ones are discussed in
\cite{DM:13,DMS:09}.
However, it is well known that not every (linear) maximal monotone
operator admits a variational formulation (i.e., grosso modo, it
cannot be represented as a bounded operator from a Banach space to its
dual), hence, in particular, not every equation admitting a mild
solution can be treated in the variational approach. The maximum
principle for mild solutions to \eqref{eq:0} obtained here is optimal,
as far as the conditions imposed on $A$ is concerned, in the sense
that we simply ask that the semigroup generated by $-A$ is positivity
preserving. Large classes of operators are thus included, beyond
second-order elliptic operators. On the other hand, we cannot treat,
at least for the time being, time-dependent random operators, some of
which are within the scope of the variational approach.
Some results on the maximum principle for mild solutions to stochastic
evolution equations can be found in the literature. One of the first
results in this direction seems to be that in \cite{Kote:comp}, where
$A$ can also be time-dependent, the coefficients $f$ and $B$ are
superposition operators, and the proof relies on some delicate
discretization arguments (unfortunately it was shown in
\cite{TessZab:Trotter} that such arguments are not entirely
correct). A different approach, again based on discretization, has
been developed in \cite{Ass:comp}.
Our proof seems conceptually quite simpler, as it relies only on
approximation arguments and a version of It\^o's formula.

As an application of the maximum principle for mild solutions to
stochastic evolution equations, we discuss the problem of positivity
of forward rates in the Heath-Jarrow-Morton framework
(see~\cite{HJM}), considering the associated Musiela SPDE on a
weighted Sobolev space (see, e.g., \cite{filipo,cm:MF10}). The same
problem is treated in \cite{FTT}, also for models driven by Poisson
random measures, by means of support theorems for Hilbert-space-valued
SDEs, under a strong smoothness assumption on the diffusion
coefficient $B$. We use instead a self-contained approximation
argument allowing to infer the positivity of mild solutions to the
Musiela SPDE in the above-mentioned weighted Sobolev space by the
positivity of mild solutions to regularized equations on weighted
$L^2$ spaces, which is in turn obtained by the maximum principle.

The rest of the text is organized as follows: in \S\ref{sec:prel},
after fixing some notation, we collect the main tools used in the
proof of the maximum principle of \S\ref{sec:main}. Basic facts about
the Heath-Jarrow-Morton model of the term structure of interest rates,
as well as about the well-posedness of Musiela's SPDE in different
function spaces, are contained in \S\ref{sec:HJM}. Two different
approximations of the volatility coefficient in Musiela's SPDE are
introduced and investigated in
\S\S\ref{sec:approx1}--\ref{sec:approx2}. These are the crucial
technical tools to deduce positivity of forward rates, discussed in
\S\ref{sec:fw}.

\medskip

\noindent%
\textbf{Acknowledgments.} The author is thankful to Fausto Gozzi for some
useful discussions on Ornstein-Uhlenbeck semigroups. The
hospitality of the Interdisziplin\"ares Zentrum f\"ur Komplexe Systeme
(IZKS) at the University of Bonn, Germany, where large part of the
work for this paper was done, is gratefully acknowledged.


\section{Preliminaries}
\label{sec:prel}

\subsection{Notation}
Let $(\Omega,\cF,\P)$ be a probability space endowed with a filtration
$(\cF_t)_{t \in \erre_+}$ satisfying the usual conditions of
right-continuity and completeness, with predictable $\sigma$-algebra
$\mathscr{P}$. All random elements will be defined on this stochastic
basis without further notice. We also stipulate that random variables
equal outside an event of probability zero are considered equal, and
that equality of stochastic processes is intended in the sense of
indistinguishability.
We shall denote by $W$ a cylindrical Wiener process on a separable
Hilbert space $U$.  

Let $H$ be a further Hilbert space. The spaces of linear bounded
operators and of Hilbert-Schmidt operators from $U$ to $H$ will be
denoted by $\cL(U,H)$ and $\cL^2(U,H)$, respectively.
Given a bilinear form $\Phi$ and a linear bounded operator $L$ on $H$,
denoting an orthonormal basis of $H$ by $(h_k)$, we set
\[
  \Tr_L \Phi := \sum_{k=1}^\infty \Phi(Lh_k,Lh_k).
\]
Given two Banach spaces $E$ and $F$, we shall write $E \embed F$ to
mean that $E$ is continuously embedded in $F$.

Let $m>0$. The space of measurable functions $\phi \colon H \to \erre$
such that
\[
  \sup_{x \in H} \frac{\abs{\phi(x)}}{1+\norm{x}^m} < \infty
\]
is denoted by $B_m(H)$. The subspace of $B_m(H)$ of continuous
functions satisfying the same boundedness condition is denoted by
$C_m(H)$. The first and second Fr\'echet derivatives of a function
$\phi \colon E \to F$ are denoted by $\phi'$ and $\phi''$,
respectively, while its G\^ateaux derivative is denoted by
$\delta\phi$. Let $k \in \enne$ and $m>0$. The space of $k$ times
continuously (Fr\'echet or, equivalently, G\^ateaux) differentiable
functions $\phi \colon H \to \erre$ such that
\[
  \sup_{x \in H} \frac{\norm{\phi^{(j)}(x)}}{1+\norm{x}^m} < \infty
  \qquad \forall j \leq k,
\]
where $\phi^{(k)}$ stands for the $k$-th order (Fr\'echet) derivative
of $\phi$, is denoted by $C^k_m(H)$. In the previous inequality, the
norm of $\phi^{(j)}(x)$ is the norm of the space of $j$-linear functions
on $H$ with values in $\erre$.

\subsection{Function spaces on $\erre_+$}
\label{ssec:sp}
For any $\alpha \in \erre$, let $\mu_\alpha$ be the measure on the
Borel $\sigma$-algebra of $\erre_+$ whose density with respect to
Lebesgue measure is equal to $x \mapsto e^{\alpha x}$. We shall denote
$L^2(\erre_+,\mu_\alpha)$ by $L^2_\alpha$. Note that
$\mu_{-\alpha}/\alpha$ is a probability measure if $\alpha>0$. We
shall use the elementary observation that, for any $\alpha>0$,
$L^2_\alpha$ is continuously embedded in $L^1(\erre_+)$. In fact, for
any $f \in L^2_\alpha$, one has, by the Cauchy-Schwarz inequality,
\[
  \int_0^\infty \abs{f(x)}\,dx = \int_0^\infty \abs{f(x)}
  e^{\frac{\alpha}{2}x} e^{-\frac{\alpha}{2}x} \,dx \leq
  \frac{1}{\sqrt{\alpha}} \biggl( \int_0^\infty \abs{f(x)}^2
  e^{\alpha x} \,dx \biggr)^{1/2}.
\]

For any $\alpha>0$, let $H_\alpha$ denote the vector space of
functions $\phi \in L^1_{\mathrm{loc}}(\erre_+)$ such that $\phi' \in
L^2_\alpha$, where $\phi'$ stands for the derivative of $\phi$ in the
sense of distributions. Since $L^2_\alpha \embed L^1(\erre_+)$, it
follows that every $\phi \in H_\alpha$ admits a finite limit at
$+\infty$. In fact, for any $a \in \erre_+$ such that
$\abs{\phi(a)}<\infty$, one has
\[
\phi(x) = \phi(a) + \int_a^x \phi'(y)\,dy,
\]
hence
\[
  \phi(\infty) := \lim_{x \to +\infty} \phi(x)
  = \phi(a) + \int_a^\infty \phi'(y)\,dy < \infty.
\]
The vector space $H_\alpha$ endowed with the scalar product
\[
\ip{\phi}{\psi}_{H_\alpha} := \phi(\infty)\psi(\infty)
+ \int_0^\infty \phi'(x)\psi'(x) e^{\alpha x}\,dx,
\]
and corresponding norm
\[
\norm[\big]{\phi}^2_{H_\alpha} :=
\phi(\infty)^2 + \int_0^\infty \abs{\phi'(x)}^2 e^{\alpha x}\,dx,
\]
is a separable Hilbert space. This is a slight modification of a
construction suggested by Filipovi\'c in \cite{filipo} (where an
equivalent norm is adopted) already used, e.g., in \cite{cm:QF10,
  tehranchi}.
The spaces $H_\alpha$ satisfy simple but crucial embedding properties.
\begin{lemma}
  Let $\alpha>0$. Then $H_\alpha$ is continuously embedded in
  $C_b(\erre_+)$ and in $L^p(\erre_+) \oplus \erre$ for every
  $p \in [1,\infty\mathclose[$.
\end{lemma}
\begin{proof}
  Let $\phi \in H_\alpha$. The Cauchy-Schwarz inequality yields
  \begin{align*}
    \abs[\big]{\phi(x) - \phi(\infty)}
    &\leq \int_x^\infty \abs[\big]{\phi'(y)}\,dy%
    = \int_x^\infty \abs[\big]{\phi'(y)} e^{\frac{\alpha}{2} y}
    e^{-\frac{\alpha}{2} y}\,dy\\
    &\leq \left( \int_0^\infty \abs[\big]{\phi'(y)}^2 e^{\alpha y}\,dy
    \right)^{1/2} \left( \int_x^\infty e^{-\alpha y}\,dy \right)^{1/2}\\
    &\leq \norm[\big]{\phi}_{H_\alpha} \frac{1}{\sqrt\alpha}
    e^{-\frac{\alpha}{2}x},
  \end{align*}
  where the last term, as a function of $x$, belongs to $L^p(\erre_+)$
  for all $p \in [1,\infty]$. The continuity of $\phi$ can be
  established by a completely similar argument.
\end{proof}

We shall also need basic properties of the translation semigroup
$S=(S(t))_{t \geq 0}$, $S(t)\colon \phi \mapsto \phi(\cdot + t)$.
\begin{lemma}
  Let $\alpha>0$. The translation semigroup $S$ is a strongly
  continuous semigroup of contractions on $H_\alpha$ with
  infinitesimal generator $-A$,
  \begin{align*}
    A\colon \dom(A) \subset H_\alpha &\longto H_\alpha\\
    \phi &\longmapsto -\phi',
  \end{align*}
  where
  $\dom(A) = \{\phi \in H_\alpha \cap C^1(\erre_+):\; \phi' \in
  H_\alpha\}$.
\end{lemma}
\begin{proof}
  The strong continuity of $S$ in $H_\alpha$, equipped with an
  equivalent norm, has been proved in \cite[pp.~78--79]{filipo}, hence
  it continues to hold in our case too. The identification of the
  negative generator $A$ is an immediate consequence thereof
  (cf. \cite[Lemma~4.2.2]{filipo}). The contractivity of $S$ is
  established as follows: for any $\phi \in H_\alpha$,
  \begin{align*}
    \norm[\big]{S(t)\phi}^2_{H_\alpha}
    &\leq \phi(\infty)^2 + \int_0^\infty \abs[\big]{\phi'(x+t)}^2
      e^{\alpha x}\,dx\\
    &= \phi(\infty)^2 
      + e^{-\alpha t} \int_t^\infty \abs[\big]{\phi'(x)}^2 e^{\alpha x}\,dx\\
    &\leq \norm[\big]{\phi}_{H_\alpha}^2.
    \qedhere
  \end{align*}
\end{proof}

The strong continuity of $S$ in $L^2_\alpha$ is most likely well
known, but we have not been able to find a reference, hence we provide the
(simple) proof for the reader's convenience.
\begin{lemma}
  Let $\alpha>0$. The translation semigroup $S$ is strongly continuous
  in $L^2_{-\alpha}$, and the semigroup $e^{-\frac{\alpha}{2} \cdot}S$
  is contractive in $L^2_{-\alpha}$.
\end{lemma}
\begin{proof}
  Let $f \in L^2_{-\alpha}$. One has, for any $t>0$,
  \begin{align*}
  \norm[\big]{S(t)f - f}_{L^2_{-\alpha}} 
  &= \norm[\big]{f(\cdot + t)e^{-\frac{\alpha}{2}\cdot} 
     - f e^{-\frac{\alpha}{2}\cdot}}_{L^2}\\
  &\leq \norm[\big]{f(\cdot + t) e^{-\frac{\alpha}{2}\cdot} 
     - f(\cdot + t) e^{-\frac{\alpha}{2}(\cdot + t)}}_{L^2}\\
  &\quad + \norm[\big]{f(\cdot + t) e^{-\frac{\alpha}{2}(\cdot + t)}
     - f e^{-\frac{\alpha}{2}\cdot}}_{L^2},
  \end{align*}
  where, since $f \in L^2_{-\alpha}$ implies that
  $fe^{\frac{\alpha}{2} \cdot} \in L^2$, the last term converges to
  zero as $t \to 0$ by the strong continuity of $S$ in $L^2$ (see,
  e.g., \cite[{\S}I.4.16]{EnNa}). Moreover,
  \begin{align*}
  &\norm[\big]{f(\cdot + t) e^{-\frac{\alpha}{2}\cdot} 
     - f(\cdot + t) e^{-\frac{\alpha}{2}(\cdot + t)}}_{L^2}\\
  &\hspace{3em} = \norm[\big]{f(\cdot + t) e^{-\frac{\alpha}{2}(\cdot+t)}
      e^{\frac{\alpha}{2}t} 
     - f(\cdot + t) e^{-\frac{\alpha}{2}(\cdot + t)}}_{L^2}\\
  &\hspace{3em} = \bigl( e^{\frac{\alpha}{2}t} - 1 \bigr)
  \norm[\big]{f(\cdot + t) e^{-\frac{\alpha}{2}(\cdot+t)}}_{L^2}\\
  &\hspace{3em} \leq \bigl( e^{\frac{\alpha}{2}t} - 1 \bigr)
  \norm[\big]{f}_{L^2},
  \end{align*}
  where the last term also converges to zero as $t \to 0$. Strong
  continuity of $S$ in $L^2_{-\alpha}$ is thus proved. The
  contractivity of $\bigl(e^{-\frac{\alpha}{2}t}S(t)\bigr)_{t\geq 0}$
  is a consequence of the following entirely analogous computation:
  \[
  \norm[\big]{S(t)f}_{L^2_{-\alpha}}
  = \norm[\big]{f(\cdot + t)e^{-\frac{\alpha}{2} \cdot}}_{L^2}
  = e^{\frac{\alpha}{2} t} \norm[\big]{f(\cdot + t)e^{-\frac{\alpha}{2}(\cdot+t)}}_{L^2}
  \leq e^{\frac{\alpha}{2} t} \norm[\big]{f}_{L^2_{-\alpha}}.
  \qedhere
  \]
\end{proof}

The lemma implies that there exists a densely defined linear operator
$A$ such that $-A$ is the infinitesimal generator of $S$, and that
$A+\alpha/2$ is maximal monotone in $L^2_{-\alpha}$. Moreover, one has
$A\phi=-\phi'$ for every $\phi$ belonging to $C^\infty_c(\erre_+)$,
which is dense in $L^2_{-\alpha}$.
We shall use the same notation for the generators of $S$ in $H_\alpha$
and in $L^2_{-\alpha}$, as they coincide on the intersection of their
domains.

\subsection{A hypoelliptic Ornstein-Uhlenbeck semigroup}
\label{ssec:OU}
Let $R=(R_\alpha)_{\alpha>0}$ be the hypoelliptic Ornstein-Uhlenbeck
semigroup on $C_b(H)$, the space of continuous bounded functions on
$H$ with values in $\erre$, defined as
\[
  R_\alpha g(x) := \int_H g\bigl( e^{\alpha C}x + y \bigr)
  \,\gamma_\alpha(dy),
\]
where $C$ is a strictly negative self-adjoint operator with $C^{-1}$ of
trace class, and $\gamma_\alpha$ is a centered Gaussian measure on $H$
with covariance operator
\[
  Q_\alpha := -\frac12 C^{-1} \bigl( I - e^{2\alpha C} \bigr).
\]
We shall need the following properties of the semigroup $R$, for the
proof of which we refer to the monographs \cite{DP:K,DZ02,FGS}.
\begin{lemma}
  \label{lm:OU}
  Let $m>0$. The semigroup $R$ extends to a semigroup of continuous
  linear operators on $C_m(H)$. Moreover, the image of $B_m(H)$
  through $R_\alpha$ is contained in $C^\infty(H)$.
\end{lemma}

The following pointwise continuity property of $\alpha \mapsto R_\alpha$
will play a fundamental role.
\begin{prop}
  \label{prop:OU}
  Let $m>0$ and $g \in C_m(H)$. Then $R_\alpha g$ converges pointwise to $g$
  as $\alpha \to 0$.
\end{prop}
\begin{cor}
  \label{cor:OU}
  Let $g \in C^1_m(H)$ and $x \in H$. Then $R_\alpha g$ is Fr\'echet
  differentiable with
  \[
    (R_\alpha g)'(x) \colon v \longmapsto \int_H g'\bigl( e^{\alpha C}x + y\bigr)
    e^{\alpha C}v \,\gamma_\alpha(dy)
  \]
  and $(R_\alpha g)'(x)v \to g'(x)v$ as $\alpha \to 0$ for every $v
  \in H$. Moreover, if $g'$ is G\^ateaux differentiable with
  \[
  \norm[\big]{\delta g'(x)}_{\cL_2(H^2;\erre)} \lesssim 1 +
  \norm{x}^m
  \]
  and $x \mapsto \delta g'(x)(v,w)$ is continuous for all $v,w
  \in H$, then $R_\alpha g$ is twice Fr\'echet differentiable with
  \[
    (R_\alpha g)''(x) \colon (v,w) \longmapsto
    \int_H \delta g'\bigl( e^{\alpha C}x + y\bigr)%
    \bigl( e^{\alpha C}v,e^{\alpha C}w \bigr)\,\,\gamma_\alpha(dy)
  \]
  and $(R_\alpha g)''(x)(v,w) \to \delta g'(x)(v,w)$ as $\alpha \to 0$
  for every $v,w \in H$.
\end{cor}
\begin{proof}
  Let $\abs{t} \in \mathopen]0,1]$ and $x$, $v \in H$. One has
  \[
  \frac{R_\alpha g(x+tv) - R_\alpha g(x)}{t} = \int_H
  \frac{g\bigl(e^{\alpha C}(x+tv) + y\bigr) %
  - g\bigl(e^{\alpha C}x + y\bigr)}{t}\,\gamma_\alpha(dy),
  \]
  where, recalling that $g \in C^1(H)$,
  \[
  \lim_{t \to 0} \frac{g\bigl(e^{\alpha C}(x+tv) + y\bigr) %
  - g\bigl(e^{\alpha C}x + y\bigr)}{t} = g'\bigl(e^{\alpha C}x + y\bigr)
  e^{\alpha C} v.
  \]
  Defining the continuously differentiable function $\tilde{g}:[0,1]
  \to \erre$ as $\tilde{g}(t):=g\bigl(e^{\alpha C}(x+tv) + y\bigr)$,
  the mean value theorem yields the existence of $t_0 \in [0,t]$
  (dependent on $t$) such that
  \[
  g\bigl(e^{\alpha C}(x+tv) + y\bigr) - g\bigl(e^{\alpha C}x + y\bigr)
  = g'\bigl(e^{\alpha C}(x+t_0v) + y\bigr) \, t \, e^{\alpha C}v,
  \]
  from which it follows, recalling that $C$ is negative-definite,
  \[
  \abs[\big]{g\bigl(e^{\alpha C}(x+tv) + y\bigr) %
    - g\bigl(e^{\alpha C}x + y\bigr)} \lesssim \abs{t} \, \bigl( 1 +
  \norm{x}^m + \norm{v}^m + \norm{y}^m \bigr) \norm{v}.
  \]
  Since $\gamma_\alpha$ has finite moments of all (positive) orders,
  the dominated convergence theorem implies that
  \[
    (R_\alpha g)'(x) \colon v \longmapsto \int_H g'\bigl( e^{\alpha C}x
    + y\bigr) e^{\alpha C}v \,\gamma_\alpha(dy).
  \]
  Writing
  \begin{align*}
  \int_H g'\bigl( e^{\alpha C}x + y\bigr) e^{\alpha C}v \,\gamma_\alpha(dy)
  &= \int_H g'\bigl( e^{\alpha C}x + y\bigr) v \,\gamma_\alpha(dy)\\
  &\quad + \int_H g'\bigl( e^{\alpha C}x + y\bigr) (e^{\alpha C}v - v)
  \,\gamma_\alpha(dy),
  \end{align*}
  the first term on the right-hand side converges to $g'(x)v$ as
  $\alpha \to 0$ by proposition~\ref{prop:OU}. It suffices then to
  show that the second term on the right-hand side converges to zero
  as $\alpha \to 0$. One has, using once again that $C$ is
  negative-definite,
  \[
  \abs[\bigg]{\int_H g'\bigl( e^{\alpha C}x + y\bigr) (e^{\alpha C}v -
    v) \,\gamma_\alpha(dy)} \lesssim \norm[\big]{e^{\alpha C}v-v}
  \int_H \bigl( \norm{x}^m + \norm{y}^m \bigr)\,\gamma_\alpha(dy),
  \]
  where the right-hand side converges to zero because the semigroup
  $\alpha \mapsto e^{\alpha C}$ is strongly continuous and the $m$-th
  moment of $\gamma_\alpha$, as a function of $\alpha$, is bounded.

  Similarly, the G\^ateaux differentiability of $g'$ implies
  \[
  \lim_{t \to 0} \frac{g'\bigl(e^{\alpha C}(x+tw)+y\bigr)e^{\alpha C}v %
  - g'\bigl(e^{\alpha C}x+y\bigr)e^{\alpha C}v}{t} %
  = \delta g'\bigl(e^{\alpha C}x+y\bigr)(e^{\alpha C}v,e^{\alpha C}w)
  \]
  for all $x, y, v, w \in H$, and, defining the differentiable
  function $\phi:[0,1] \to \erre$ as
  \[
  \phi(t) := g'\bigl(e^{\alpha C}(x+tw)+y\bigr)e^{\alpha C}v,
  \]
  the fundamental theorem of calculus yields
  \begin{align*}
    &g'\bigl(e^{\alpha C}(x+tw)+y\bigr)e^{\alpha C}v %
    - g'\bigl(e^{\alpha C}x+y\bigr)e^{\alpha C}v\\
    &\hspace{3em} = \phi(t)-\phi(0) = \int_0^t \delta
    g'\bigl(e^{\alpha C}(x+sw)+y\bigr)(e^{\alpha C}v,e^{\alpha C}w)\,ds,
  \end{align*}
  hence, as $C$ is negative definite,
  \begin{align*}
    &\abs[\big]{g'\bigl(e^{\alpha C}(x+tw)+y\bigr)e^{\alpha C}v %
      - g'\bigl(e^{\alpha C}x+y\bigr)e^{\alpha C}v}\\
    &\hspace{3em} \lesssim \abs{t} \, \bigl( 1 + \norm{x}^m +
    \norm{w}^m + \norm{y}^m \bigr) \norm{v} \, \norm{w}.
  \end{align*}
  The dominated convergence theorem then implies
  \begin{align*}
  (R_\alpha g)''(x)(v,w) %
  &= \lim_{t \to 0} \frac{(R_\alpha g)'(x+tw)v - (R_\alpha g)'(x)v}{t}\\
  &= \int_H \delta g'\bigl( e^{\alpha C}x + y\bigr)%
    \bigl( e^{\alpha C}v,e^{\alpha C}w \bigr)\,\,\gamma_\alpha(dy).
  \end{align*}
  Furthermore, writing $e^{\alpha C}v=v+e^{\alpha C}v-v$ and similarly
  for $e^{\alpha C}w$, it follows by bilinearity of $\delta g'$ that
  \begin{align*}
  \delta g'\bigl( e^{\alpha C}x + y\bigr)(e^{\alpha C}v,e^{\alpha C}w)
  &= \delta g'\bigl( e^{\alpha C}x + y\bigr)(v,w)\\
  &\quad + \delta g'\bigl( e^{\alpha C}x + y\bigr)(v,e^{\alpha C}w-w)\\
  &\quad + \delta g'\bigl( e^{\alpha C}x + y\bigr)(e^{\alpha C}v-v,w)\\
  &\quad + \delta g'\bigl( e^{\alpha C}x + y\bigr)(e^{\alpha C}v-v,e^{\alpha C}w-w).
  \end{align*}
  Since $x \mapsto \delta g'(x)(v,w) \in C_m(H)$ for all $v,w \in H$,
  proposition~\ref{prop:OU} implies
  \[
  \int_H \delta g'\bigl( e^{\alpha C}x + y\bigr)(v,w)
  \,\gamma_\alpha(dy) \longto \delta g'(x)(v,w)
  \]
  as $\alpha \to 0$. It remains to show that the integrals over
  $H$ with respect to $\gamma_\alpha$ of the other three terms on the
  right-hand side of the previous identity converge to zero as $\alpha
  \to 0$. It follows by the polynomial boundedness of $\delta g'$ and
  the contractivity of $e^{\alpha C}$ that
  \begin{align*}
    &\abs[\bigg]{\int_H \delta g'\bigl( e^{\alpha C}x + y\bigr)%
      (e^{\alpha C}v-v,e^{\alpha C}w-w) \,\gamma_\alpha(dy)}\\
    &\hspace{3em} \lesssim \norm[\big]{e^{\alpha C}v-v} \,
    \norm[\big]{e^{\alpha C}w-w} \int_H \bigl( 1 + \norm{x}^m +
    \norm{y}^m \bigr) \,\gamma_\alpha(dy),
  \end{align*}
  where the right-hand side converges to zero as $\alpha \to 0$
  because, as before, $e^{\alpha C}$ is strongly continuous and the
  $m$-th moment of $\gamma_\alpha$ is bounded with respect to
  $\alpha$. The remaining two terms can be treated in a completely
  similar way.
\end{proof}

\subsection{A differentiability result}
The following result is most likely well known, but since we could not
locate a proof in the literature, we include a proof for the
convenience of the reader.
\begin{lemma}
  \label{lm:gigi}
  Let $E=L^2(X,\mathscr{A},m)$, $g \in C^2(\erre)$ with $g(0)=g'(0)=0$ and
  $g''\in C_b(\erre)$, and $G:E \to \erre$ be defined as
  \[
  G: u \longmapsto \int_X g(u) \,dm.
  \]
  Then $G$ is continuously Fr\'echet differentiable on $E$, with
  Fr\'echet derivative at $u \in E$ given by
  \[
    G'(u): v \longmapsto \ip{g'(u)}{v} = \int_X g'(u)v \,dm.
  \]
  Moreover, $G'$ is G\^ateaux differentiable on $E$, with G\^ateaux
  derivative at $u \in E$ given by
  \[
    \delta G'(u): (v,w) \longmapsto \ip{g''(u)}{vw}
    = \int_X g''(u)vw \,dm,
  \]
  and $(u,v,w) \mapsto \delta G'(u)(v,w)$ is continuous.
\end{lemma}
\begin{proof}
  Let us first show that $G$, as well as the (for now just formal)
  expressions for $G'$ and $\delta G'$ are well defined: since
  $g'(0)=0$ and $g''$ is bounded, the fundamental theorem of calculus
  yields $\abs{g'(x)} \lesssim \abs{x}$ for all $x \in \erre$, hence
  also, as $g(0)=0$, $\abs{g(x)} \lesssim \abs{x}^2$. This immediately
  implies that the integral defining $G(u)$ is finite for every
  $u \in E$. Similarly, in view of the Cauchy-Schwarz inequality, the
  linear bound on $g'$ implies that $G'(u)$ is a continuous linear map
  from $E$ to $\erre$ for every $u \in E$, and the boundedness of
  $g''$ implies that $\delta G'(u)$ is a continuous bilinear map from
  $E \times E$ to $\erre$ for every $u \in E$.

  In order to show the Fr\'echet differentiability of $G$ at $u$,
  note that, for any $v \in E$, one has
  \[
    G(u+v) - G(u) - \ip{g'(u)}{v}
    = \ip[\bigg]{v}{\int_0^1 \bigl(g'(u+tv) - g'(u)\bigr)\,dt},
  \]
  hence
  \[
    \norm[\big]{G(u+v) - G(u) - \ip{g'(u)}{v}} \leq \norm{v} \,
    \norm[\bigg]{\int_0^1 \bigl(g'(u+tv) - g'(u)\bigr)\,dt}.
  \]
  Taking into account that $g''$ is bounded, elementary calculus shows
  that the second term on the right-hand side is bounded by
  $\norm{g''}_{C(\erre)} \, \norm{v}$, so that
  \[
    \frac{\norm{G(u+v) - G(u) - \ip{g'(u)}{v}}}{\norm{v}} \longrightarrow 0
  \]
  as $v \to 0$ in $E$, i.e. the Fr\'echet derivative of $G$ at
  $u \in E$ is indeed the linear map $G'(u): v \mapsto
  \ip{g'(u)}{v}$. Let us prove its continuity: by the isomorphism
  $\cL(E,\erre)=E' \simeq E$, it is enough to show that, for any
  sequence $(u_n) \subset E$ converging to $u$ in $E$, $g'(u_n)$
  converges to $g'(u)$ in $E$. But this follows immediately, in
  analogy to an argument already used, by the Lipschitz continuity of
  $g'$. We have thus established that $G \in C^1(E)$.
  To prove that $G'$ is G\^ateaux differentiable, let us write, for
  any $u$, $v$, $w \in E$ and $t>0$,
  \[
    \frac{G'(u+tv)w - G'(u)w}{t} - \ip{g''(u)}{vw} = \int_X \Bigl(
    \frac{g'(u+tv) - g'(u)}{t} - g''(u)v \Bigr) w \,dm,
  \]
  from which it follows, by the Cauchy-Schwarz inequality,
  \[
    \norm[\bigg]{w \mapsto \frac{G'(u+tv)w - G'(u)w}{t} -
      \ip{g''(u)}{vw}}_{\cL(E,\erre)}
    \leq \int_X \abs[\Big]{\frac{g'(u+tv) - g'(u)}{t} - g''(u)v}^2\,dm.
  \]
  Since $g'' \in C_b(\erre)$, hence $g'$ is Lipschitz continuous, the
  integrand on the right-hand side is bounded above by $\abs{v}^2$,
  modulo a multiplicative constant, and $v \in L^2(m)$. Therefore the
  dominated convergence theorem immediately yields that $G'$ is
  G\^ateaux differentiable on $E$ with G\^ateaux derivative $\delta
  G'(u): (v,w) \mapsto \ip{g''(u)}{vw}$. Let $u_n \to u$, $v_n \to v$,
  and $w_n \to w$ in $H$ as $n \to \infty$. Then $u_n \to u$ in
  $m$-measure, and $g''(u_n) \to g''(u)$ by the continuous mapping
  theorem. Setting $\bar{v}_n:=g''(u_n)v_n$ and $\bar{v}=g''(u)v$,
  this implies that $\bar{v}_n \to v$ in $L^2(m)$, hence in
  $m$-measure. In fact,
  \[
  \norm[\big]{\bar{v}_n - v}_{L^2(m)} \leq
  \norm[\big]{g''(u_n)v_n - g''(u_n)v}_{L^2(m)}
  + \norm[\big]{g''(u_n)v - g''(u)v}_{L^2(m)},
  \]
  where
  \[
  \norm[\big]{g''(u_n)v_n - g''(u_n)v}_{L^2(m)}
  \leq \norm[\big]{g''}_{C(X)} \norm[\big]{v_n - v}_{L^2(m)} \to 0,
  \]
  and $g''(u_n)v - g''(u)v \to 0$ in $m$-measure and $\abs{g''(u_n)v -
    g''(u)v} \leq \norm{g''}_{C(X)} \abs{v} \in L^2(m)$, hence
  \[
  \norm[\big]{g''(u_n)v - g''(u)v}_{L^2(m)} \to 0
  \]
  by the dominated convergence theorem. Then we have
  \[
  \abs[\big]{\delta G'(u_n)(v_n,w_n) - \delta G'(u)(v,w)}
  = \abs[\big]{\ip{\bar{v}_n}{w_n} - \ip{\bar{v}}{w}} \to 0,
  \]
  thus completing the proof.
\end{proof}
\begin{rmk}
  The mapping $G$ is \emph{never} twice Fr\'echet differentiable,
  unless $g(x)=ax^2$ for some constant $a \in \erre$. In other words,
  $G$ is twice Fr\'echet differentiable if and only if it is
  proportional to the square of the norm. This is an immediate
  consequence of the fact that the superposition operator on $L^2(m)$
  associated to a function $\phi \in C^\infty(\erre)$ is Fr\'echet
  differentiable if and only if $\phi$ is linear. In fact, identifying
  $\cL(E,\erre)$ with $E$, the derivative $G'$ can be identified with
  the superposition operator on $E$ associated to the function
  $g':\erre \to \erre$.
\end{rmk}

\subsection{Approximation and convergence in locally Lipschitzian
  SPDEs}
Consider the stochastic evolution equation on the Hilbert space $H$
\begin{equation}
  \label{eq:cozza}
du + Au\,dt + f(u)\,dt = B(u)\,dW, \qquad u(0)=u_0,
\end{equation}
where $A$ is a linear maximal monotone operator on $A$, $u_0$ is an
$\cF_0$-measurable $H$-valued random variable, and $f$, $B$ satisfy
the same measurability conditions as in (b) of \S\ref{sec:main} below.

For any Banach space $E$ and any sequence of maps $(F_n)$,
$F_n: \Omega \times \erre_+ \times H \to E$, we shall say that $(F_n)$
is uniformly locally Lipschitz continuous if for every
$R \in \erre_+$ there exists a constant $N=N(R)$, independent of $n$,
such that
\begin{equation}
  \label{eq:loli}
  \norm[\big]{F_n(\omega,t,x) - F_n(\omega,t,y)}_E \leq N \norm{x-y}
  \qquad \forall x, y \in B_R(H)
\end{equation}
and there exists $a \in H$ such that
$(\omega,t,n) \mapsto F_n(\omega,t,a)$ is bounded. If the sequence
$(F_n)$ reduces to a singleton $F$, we say that $F$ is locally
Lipschitz continuous.

If $f: \Omega \times \erre_+ \times H \to H$ and
$B: \Omega \times \erre_+ \times H \to \cL^2(U,H)$ are locally
Lipschitz continuous, there exists a unique local mild solution $u$ to
\eqref{eq:cozza} with lifetime $T$ (i.e., $T$ is a stopping time such
that the norm of $u(t)$ tends to infinity as $t$ tends to $T$ from the
left).
Furthermore, let $(f_n)$ and $(B_n)$ be sequences of locally Lipschitz
continuous maps with the same domains and codomains of $f$ and $B$,
respectively, and $(u_{0n})$ a sequence of $\cF_0$-measurable
$H$-valued random variables.  Let $u_n$, with lifetime $T_n$, be the
unique local mild solution to the equation obtained replacing $f$,
$B$, and $u_0$ in \eqref{eq:cozza} with $f_n$, $B_n$, and $u_{0n}$,
respectively.
One has the following convergence result, proved (in a more general
setting) in \cite{KvN2}.
\begin{thm}
  \label{thm:cozza}
  Assume that
  \begin{itemize}
  \item[\emph{(i)}] $(f_n)$ and $(B_n)$ are uniformly locally
    Lipschitz continuous;
  \item[\emph{(ii)}] $f_n$ and $B_n$ converge pointwise to $f$ and
    $B$, respectively, as $n \to \infty$;
  \item[\emph{(iii)}] $u_{0n} \to u_0$ in probability as $n \to \infty$.
  \end{itemize}
  Then
  \[
    u_n \ind{[\![0,T \wedge T_n]\!]} \longrightarrow
    u \ind{[\![0,T]\!]}
  \]
  in $L^0(\Omega \times \erre_+;H)$ as $n \to \infty$.
\end{thm}


\section{Main result}
\label{sec:main}
The following hypotheses on the data of \eqref{eq:0} will be in force
throughout this section:
\begin{itemize}
\item[(a)] $A$ is the generator of a positive $C_0$-semigroup $S$ on
  $H$.
\item[(b)] The mappings $f: \Omega \times \erre_+ \times H \to H$ and
  $B: \Omega \times \erre_+ \times H \to \cL^2(U,H)$ are measurable
  and adapted, and there exist constants $C_f$ and $C_B$ such that
\begin{align*}
  \norm[\big]{f(\omega,t,h_1)-f(\omega,t,h_2)}
  &\leq C_f \norm{h_1-h_2},\\
  \norm[\big]{f(\omega,t,h)}
  &\leq C_f \bigl( 1 + \norm{h} \bigr),\\
  \norm[\big]{B(\omega,t,h_1) - B(\omega,t,h_2)}_{\cL^2(U,H)}
  &\leq C_B \norm{h_1-h_2},\\
  \norm[\big]{B(\omega,t,h)}_{\cL^2(U,H)}
  &\leq C_B \bigl( 1 + \norm{h} \bigr)
\end{align*}
for all $\omega \in \Omega$, $t \in \erre_+$, and $h,h_1,h_2 \in H$.
\item[(c)] One has, for every $\omega \in \Omega$, $t \in \erre_+$ and
  $h \in H$,
\[
  - \ip[\big]{h^-}{f(\omega,t,h)} + 
  \norm[\big]{\ind{\{h \leq 0\}}B(\omega,t,h)}^2_{\cL^2(U,H)} \lesssim
  \norm[\big]{h^-}^2
\]
\end{itemize}

By classical results on the well-posedness in the mild sense of
stochastic evolution equations, it is known that, assuming $u_0 \in
L^p(\Omega,\cF_0,\P)$ for some $p \geq 0$, for any $T>0$ there exists
a unique measurable, adapted, continuous (hence predictable) process
\[
u \in L^p(\Omega;C([0,T];H)
\]
such that
\begin{gather*}
  s \mapsto S(t-s)f(s,u(s)) \in L^1(0,t;H),\\
  s \mapsto S(t-s)B(s,u(s)) \in L^2(0,t;\cL^2(U,H)),
\end{gather*}
and
\[
u(t) = S(t)u_0 + \int_0^t S(t-s)f(s,u(s))\,ds + \int_0^t
S(t-s)B(s,u(s))\,dW(s)
\]
$\P$-a.s. for all $t \in [0,T]$.

\medskip

We can now formulate the main result of this section.
\begin{thm}
\label{thm:pos}
Let $u$ be a mild solution to \eqref{eq:0}. If $u_0$ is positive, then
$u(t)$ is positive for all $t \in \erre_+$.
\end{thm}
\begin{proof}
  As a first step, we assume that $A$ is a bounded operator, so that
  $u$ is in fact a strong solution, in particular a semimartingale,
  that $\ip{Ah}{h^-} \leq 0$ for every $h \in H$, and that
  $\E\norm{u_0}^2<\infty$.
  Let $g$ and $G$ be as in lemma~\ref{lm:gigi}, and
  $G_\alpha := R_\alpha G$, where $(R_\alpha)$ is the hypoelliptic
  Ornstein-Uhlenbeck semigroup introduced in
  \S\ref{ssec:OU}. Recalling that $G_\alpha \in C^\infty(H)$ by
  lemma~\ref{lm:OU}, It\^o's formula yields,
  \begin{equation}
    \label{eq:Ito}
  \begin{split}
    G_\alpha(u) + \int_0^\cdot \ip{Au}{G'_\alpha(u)}\,ds
    &= G_\alpha(u_0)
       + \int_0^\cdot \ip{f(u)}{G'_\alpha(u)}\,ds
       + \int_0^\cdot G_\alpha'(u) B(u) \,dW\\
    &\quad + \frac12 \int_0^\cdot \Tr_{B(u)} G_\alpha''(u)\,ds.
    \end{split}
  \end{equation}
  We are going to pass to the limit as $\alpha \to 0$ in each term of
  this identity. Proposition~\ref{prop:OU} implies that
  \begin{alignat*}{3}
    G_\alpha(u_0) &\longto G(u_0),\\
    G_\alpha(u(t)) &\longto G(u(t)) &\quad \forall t>0.
  \end{alignat*}
  Moreover, $G'_\alpha(u(s)) \to G'(u(s))$ weakly for every $s \geq 0$
  by corollary~\ref{cor:OU}, hence
  \begin{align*}
    \ip[\big]{Au(s)}{G'_\alpha(u(s))}
    &\longto \ip[\big]{Au(s)}{G'(u(s))}\\
    \ip[\big]{f(u(s))}{G'_\alpha(u(s))}
    &\longto \ip[\big]{f(u(s))}{G'(u(s))}
  \end{align*}
  for every $s \geq 0$. The boundedness of $g''$ implies that $G'$
  grows at most linearly, therefore, recalling that $C$ is
  negative-definite,
  \[
    \norm{G'_\alpha(u(s))} \leq \int_H \norm[\big]{G'(e^{\alpha C}
      u(s)+y)}\,\gamma_\alpha(dy)
    \lesssim \norm{u(s)} + \int_H \norm{y}\,\gamma_\alpha(dy).
  \]
  Since the last term on the right-hand side is finite for every
  $\alpha>0$, $A$ is bounded, and $f$ also grows at most linearly, one
  has
  \[
    \abs[\big]{\ip[\big]{Au(s)}{G'_\alpha(u(s))}}
    + \abs[\big]{\ip[\big]{f(u(s))}{G'_\alpha(u(s))}} \lesssim_\alpha
    1 + \norm{u(s)}^2.
  \]
  As $u$ is pathwise continuous, the dominated convergence theorem
  yields, for every $t \geq 0$,
  \begin{align*}
    \int_0^t \ip[\big]{Au(s)}{G'_\alpha(u(s))}\,ds 
    &\longto \int_0^t \ip[\big]{Au(s)}{G'(u(s))}\,ds\\
    \int_0^t \ip[\big]{f(u(s))}{G'_\alpha(u(s))}\,ds
    &\longto \int_0^t \ip[\big]{f(u(s))}{G'(u(s))}\,ds.
  \end{align*}
  In order to have
  \[
    \int_0^t G'_\alpha(u(s)) B(u(s))\,dW(s) \longto
    \int_0^t G'(u(s)) B(u(s))\,dW(s)
  \]
  in probability, it suffices to show that
  \[
    \int_0^t
    \norm[\big]{G'_\alpha(u)B(u)-G'(u)B(u)}^2_{\cL^2(U,\erre)}\,ds \longto 0
  \]
  in probability. Since $G'_\alpha(h) \to G'(h)$ weakly for every
  $h \in H$, i.e. the convergence takes place in $\cL(H,\erre)$, ideal
  properties of the space of Hilbert-Schmidt operators imply that
  \[
    \norm[\big]{G'_\alpha(u(s))B(u(s))-G'(u(s))B(u(s))}^2_{\cL^2(U,\erre)}
    \longto 0
  \]
  for every $s \in [0,t]$. We can thus conclude once again by the
  dominated convergence theorem, as
  \begin{align*}
    \norm[\big]{G'_\alpha(u)B(u)-G'(u)B(u)}_{\cL^2(U,\erre)}
    &\leq \norm[\big]{G'_\alpha(u)-G'(u)}
      \norm[\big]{B(u)}_{\cL^2(U,H)}\\
    &\lesssim_\alpha 1 + \norm{u}^2.
  \end{align*}
  Finally, denoting a complete orthonormal basis of $U$ by $(e_k)$,
  corollary~\ref{cor:OU} also implies
  \[
    G''_\alpha(u(s))\bigl(B(u(s))e_k,B(u(s))e_k\bigr)
    \longto \delta G'(u(s))\bigl(B(u(s))e_k,B(u(s))e_k\bigr)
  \]
  for every $s \in [0,t]$ and $k \in \enne$. Moreover, since $g''$ is
  bounded and $C$ is negative-definite, lemma~\ref{lm:gigi} implies
  \begin{align*}
    \abs[\big]{G''_\alpha(u)(v,w)}
    &= \abs[\bigg]{\int_H \delta G'(e^{\alpha C}u + y)
      (e^{\alpha C}v,e^{\alpha C}w) \, \gamma_\alpha(dy)}\\
    &\leq \int_H \norm{g''}_{C(\erre)} \norm{v} \, \norm{w}\, \gamma_\alpha(dy)
      \lesssim \norm{v} \, \norm{w}
  \end{align*}
  for every $v$, $w \in H$. Therefore, since $B(u(s)) \in \cL^2(U,H)$,
  the dominated convergence theorem yields
  \[
    \int_0^t \Tr_{B(u)} G''_\alpha(u)\,ds \longto \int_0^t \Tr_{B(u)}
    \delta G'(u)\,ds
  \]
  for every $t \geq 0$. We have thus shown that
  \begin{equation}
    \label{eq:ItoG}
    \begin{split}
    G(u) + \int_0^\cdot \ip{Au}{G'(u)}\,ds
    &= G(u_0)
      + \int_0^\cdot \ip{f(u)}{G'(u)}\,ds
      + \int_0^\cdot G'(u) B(u) \,dW\\
    &\quad + \frac12 \int_0^\cdot \Tr_{B(u)} \delta G'(u)\,ds.
    \end{split}
  \end{equation}
  Let $(g_n) \subset C^2(\erre)$ be a sequence of functions such that
  $g_n(0)=g'_n(0)=0$ for all $n \in \enne$, $(g_n'')$ is uniformly
  bounded, and
  \[
    g_n(x) \to \frac12 \abs{x^-}^2, \qquad g_n'(x) \to -x^-, \qquad
    g''_n(x) \to \ind{\{x \leq 0\}}
  \]
  for every $x \in \erre$. Defining the sequence of maps $(G_n)$ as
  \begin{align*}
    G_n \colon H &\longrightarrow \erre\\
    v &\longmapsto \int_X g_n(v(x))\,m(dx),
  \end{align*}
  \eqref{eq:ItoG} holds with $G$ replaced by $G_n$ for each
  $n \in \enne$, taking the limit as $n \to \infty$ yields
  \begin{align*}
    \frac12 \norm{u^-}^2 - \int_0^\cdot \ip{Au}{u^-}\,ds
    &= \frac12 \norm{u_0^-}^2 - \int_0^\cdot \ip{f(u)}{u^-}\,ds
      + \int_0^\cdot u^- B(u)\,dW\\
    &\quad + \frac{1}{2} \int_0^\cdot \Tr_{B(u)} \ind{\{u \leq 0\}}\,ds,
  \end{align*}
  where
  \[
    \Tr_{B(u)} \ind{\{u \leq 0\}} = \norm[\big]{\ind{\{u \leq 0\}}
      B(u)}^2_{\cL^2(U,H)}.
  \]
  Since $\ip{Ah}{h^-} \leq 0$ for every $h \in H$ by assumption, the
  second term on the left-hand side is positive. Moreover, thanks to
  assumption (c), the sum of the second and fourth term on the
  right-hand side can be estimated by
  \[
    C \int_0^\cdot \norm{u^-}^2\,ds,
  \]
  with $C$ a positive constant.  We are thus left with
  \begin{equation}
    \label{eq:toppa}
    \norm{u^-}^2 \leq \norm{u_0^-}^2
    + C \int_0^\cdot \norm{u^-}^2\,ds + \int_0^\cdot u^- B(u) dW.
  \end{equation}
  Let $(T_n)$ be a localizing sequence for the continuous local
  martingale on the right-hand side \eqref{eq:toppa}. Introducing the
  stopped process $u_n := u^{T_n}$, one has
  \begin{align*}
    \norm[\big]{u_n^-}^2
    &\leq \norm[\big]{u_0^-}^2
      + C \int_0^{\cdot \wedge T_n} \norm[\big]{u^-}^2
      + \int_0^{\cdot \wedge T_n} u^- B(u) \,dW\\
    &\leq \norm[\big]{u_0^-}^2
      + C \int_0^\cdot \norm[\big]{u_n^-}^2
      + \int_0^{\cdot \wedge T_n} u^- B(u) \,dW.
  \end{align*}
  Recalling that $\E\norm{u_0}^2<\infty$ by assumption, taking
  expectation on both sides and applying Tonelli's theorem yields
  \[
    \E\norm[\big]{u_n^-}^2 \leq \E\norm[\big]{u_0^-}^2
    + C \int_0^\cdot \E\norm[\big]{u_n^-}^2,
  \]
  hence also, by Gronwall's inequality,
  \[
    \E\norm[\big]{u_n^-(t)}^2 \lesssim_t \E\norm[\big]{u_0^-}^2 \qquad
    \forall t \in \erre_+.
  \]
  Passing to the limit as $n \to \infty$, Fatou's lemma yields
  \[
    \E\norm[\big]{u^-(t)}^2 \lesssim_t \E\norm[\big]{u_0^-}^2 \qquad
    \forall t \in \erre_+.
  \]
  The claim is thus proved under the additional assumptions that $A$
  is bounded, that $\ip{Ah}{h^-}\leq 0$ for every $h \in H$, and
  that $u_0$ has finite second moment.
  We are now going to show that the result continues to hold also when
  these additional assumptions are not satisfied.

  Let us assume that $A$ is unbounded, and introduce its Yosida
  approximation
  \[
    A_\lambda := \frac{1}{\lambda} \bigl( I-(I+\lambda A)^{-1} \bigr),
    \qquad \lambda>0.
  \]
  It is well known that $A_\lambda$ is a bounded linear monotone
  operator. Let us show that the positivity preserving property of the
  semigroup $S$ implies that $\ip{A_\lambda h}{h^-} \leq 0$ for every
  $h \in H$ and $\lambda>0$: one has
  \[
    \lambda \ip[\big]{A_\lambda h}{h^-}
    = \ip[\big]{h}{h^-} - \ip[\big]{(I+\lambda A)^{-1}h}{h^-},
  \]
  where $\ip{h}{h^-} \leq -\norm{h^-}^2$ and, recalling that the
  resolvent of $A$ is a positivity preserving contraction of $H$,
  \begin{align*}
  -\ip[\big]{(I+\lambda A)^{-1}h}{h^-}
  &= - \ip[\big]{(I+\lambda A)^{-1}h^+}{h^-}
  + \ip[\big]{(I+\lambda A)^{-1}h^-}{h^-}\\
  &\leq \norm{h^-}^2,
  \end{align*}
  thus establishing the claim.
  Let us now consider the regularized equation
  \[
    du_\lambda + A_\lambda u_\lambda\,dt = f(u_\lambda)\,dt + B(u_\lambda)\,dW,
    \qquad u_\lambda(0)=u_0,
  \]
  for which the first part of the proof implies that
  $u_\lambda \geq 0$. By virtue of the assumption that $u_0$ has
  finite second moment, one has, for any $T>0$,
  \[
    \lim_{\lambda \to 0} \E\sup_{t \leq T}
    \norm[\big]{u_\lambda(t) - u(t)}^2 = 0,
  \]
  where $u$ is the (unique) mild of \eqref{eq:0} (see, e.g.,
  \cite{cm:JFA13}), from which it follows that the positivity of $u_0$
  implies the positivity of $u(t)$ for all $t \geq 0$.

  Finally, we can remove the assumption that $u_0 \in L^2(\P)$. In
  fact, let $u_{0n}:=\ind{\{\norm{u_0}\leq n\}} u_0 \in L^2(\P)$, and
  denote the unique mild solution to \eqref{eq:0} with initial
  condition $u_{0n}$ by $u_n$. Then $u_n(t) \geq 0$ for all $t \geq 0$
  by the previous reasoning, and theorem~\ref{thm:cozza} implies that
  $u_n \to u$ in $L^0(\Omega \times [0,T];H)$ for every $T>0$. This in
  turn implies that $u(t) \geq 0$ for all $t \geq 0$, and the proof is
  completed.
\end{proof}


\section{Heath-Jarrow-Morton models and Musiela's equation}
\label{sec:HJM}
As an application of the abstract positivity-preserving property of
the previous section, we are going to provide sufficient conditions
for the positivity of forward rates in the Heath-Jarrow-Morton (HJM)
model (see~\cite{HJM}). We first briefly recall its origin and the
reparametrization introduced by Musiela in \cite{musiela}.

Denoting the forward rate at time $t$ for date
$T \geq t$ by $f(t,T)$, (a version of) the HJM model assumes that
\begin{equation}
  \label{eq:f}
  f(t,T) = f(0,T) + \int_0^t \bar{\alpha}(s,T)\,ds
  + \sum_{k=1}^\infty \int_0^t \bar{\sigma}_k(s,T)\,dw^k(s),
\end{equation}
where $(w^k)$ is a sequence of independent standard Wiener processes,
$f(0,T)$ is an $\cF_0$-measurable random variable, and
$\bar{\alpha}(\cdot,T)$, $(\bar{\sigma}_k(\cdot,T))$ are predictable
processes such that
\[
  \int_0^T \abs[\big]{\bar{\alpha}(s,T)}
  + \norm[\big]{\bar{\sigma}(s,T)}^2_{\ell^2}\,ds < \infty
\]
$\P$-almost surely. One of the major results of $\cite{HJM}$ is that
the discounted bond price process $\check{B}(\cdot,T)$ implied by the
forward rates $f(\cdot, T)$, i.e.
\[
  \check{B}(t,T) = \exp\biggl(-\int_0^t f(s,s)\,ds
  - \int_t^T f(t,s)\,ds \biggr),
\]
is a local martingale (with respect to $\P$) if and only if
\begin{equation}
  \label{eq:HJM}
  \bar{\alpha}(t,T) = \sum_{k=1}^\infty \bar{\sigma}_k(t,T)
  \int_t^T \bar{\sigma}_k(t,s)\,ds.
\end{equation}
Musiela observed in \cite{musiela} that a simple change of variable
allows to write \eqref{eq:f}, interpreted as a family of processes
indexed by $T \in \erre_+$, as the mild solution to a first-order
stochastic PDE. In particular, setting $x:=T-t$ (which corresponds to
considering the time \emph{to} maturity rather than the time \emph{of}
maturity), \eqref{eq:f} can be written as
\[
  f(t,t+x) = f(0,t+x) + \int_0^t \bar{\alpha}(s,t+x)\,ds
  + \sum_{k=1}^\infty \int_0^t \bar{\sigma}_k(s,t+x)\,dw^k(s).
\]
Introducing (for now in a purely formal way) the family of shift
operators $(S(t))_{t \in \erre_+}$ as
\begin{align*}
  S(t) \colon \erre^\erre &\longto \erre^\erre\\
  \phi(\cdot) &\longmapsto \phi(\cdot+t),
\end{align*}
and setting, for any function $\phi:\erre^2 \to \erre$,
$S(t)\phi(x,y):=\phi(x,y+t)$, \eqref{eq:f} can also be written as
\[
  f(t,t+x) = S(t)f(0,x) + \int_0^t S(t-s) \bar{\alpha}(s,s+x)\,ds
  + \sum_{k=1}^\infty \int_0^t S(t-s) \bar{\sigma}_k(s,s+x)\,dw^k(s).
\]
If there exists a Hilbert space of functions $H$ such that the family
of shift operators $S(t)_{t \geq 0}$ is a strongly continuous
semigroup and, on an event of probability one,
$x \mapsto f(t,t+x) \in H$,
$x \mapsto \alpha(\cdot,\cdot+x) \in L^1(0,t;H)$, and
$x \mapsto \sigma(\cdot,\cdot+x) \in L^2(0,t;\ell^2(H))$, setting
\[
  u(t) := f(t,t+\cdot), \qquad
  \alpha(t) := \bar{\alpha}(t,t+\cdot), \qquad
  \sigma(t) := \bar{\sigma}(t,t+\cdot),
\]
one has
\begin{equation}
  \label{eq:u}
  u(t) = S(t)u(0) + \int_0^t S(t-s) \alpha(s)\,ds
  + \sum_{k=1}^\infty \int_0^t S(t-s) \sigma_k(s)\,dw^k(s),
\end{equation}
i.e., in differential notation,
\[
du + Au\,dt = \alpha(t)\,dt + \sum_k \sigma_k(t)\,dw^k(t),
\]
where $-A$ is the infinitesimal generator of the translation semigroup
$S$, hence an operator that acts on smooth functions as a first
derivative.

It is by now well known that Hilbert spaces of functions indeed exist
that permit to make the above reasoning rigorous. A particularly
convenient choice of forward curves is the class of Hilbert spaces
$H_\alpha$, $\alpha>0$, defined in \S\ref{sec:prel}. In fact, elements
of $H_\alpha$ are continuous functions, and constant functions belong
to $H_\alpha$. These properties are essential, as empirical
observations suggest that forward curves are smooth (at least
continuous) and tend to flatten out for large times to maturity
without necessarily decaying to zero.
For technical reasons that will become apparent later, we shall also
consider as state space the Hilbert space $L^2_{-\alpha}$ defined in
\S\ref{ssec:sp}.

In order to consider HJM models (in Musiela's parametrization) for
which the diffusion coefficient depends on the forward curve itself,
we are naturally led to consider stochastic evolution equations, on
the Hilbert spaces $H_\alpha$ and $L^2_{-\alpha}$, $\alpha>0$, of the
form
\begin{equation}
  \label{eq:ugo}
  du + Au\,dt = \beta(t,u)\,dt + \sigma(t,u)\,dW,
  \qquad u(0)=u_0,
\end{equation}
to which we shall refer as Musiela's equation. Here $W$ is a
cylindrical Wiener process on a separable Hilbert space $U$, as in
\S\ref{sec:main}, and, writing $H$ for either $H_\alpha$ or
$L^2_{-\alpha}$, 
$\sigma:\Omega \times \erre_+ \times H \to \cL^2(U,H)$
satisfies measurability, Lipschitz-continuity, and linear growth
assumptions completely analogous to those imposed on $B$ in
\S\ref{sec:main}, and
$\beta: \Omega \times \erre_+ \times H \to H$ is such
that, at least formally,
\[
\beta(t,v) = \ip[\big]{\sigma(t,v)}{I\sigma(t,v)}_U,
\]
where, for any separable Hilbert space $K$, $I$ is the operator
defined by
\begin{align*}
  I\colon L^1_{\mathrm{loc}}(\erre_+;K)
  &\longrightarrow L^1_{\mathrm{loc}}(\erre_+;K)\\
  v &\longmapsto \int_0^\cdot v(y)\,dy.
\end{align*}
In other words, denoting a complete orthonormal basis of $U$ by
$(e_k)$ and setting $\sigma_k(\cdot)=\ip{\sigma(\cdot)}{e_k}$, one has
\[
\beta(t,v) = \sum_{k=1}^\infty \sigma_k(t,v) I\sigma_k(t,v) =
\sum_{k=1}^\infty \sigma_k(t,v) \int_0^\cdot
\bigl[\sigma_k(t,v)\bigr](y)\,dy.
\]

\subsection{Well-posedness on $H_\alpha$}
It is not difficult to show (cf.~\cite[Corollary~5.1.2]{filipo}) that
the map $h \mapsto hIh$ is a locally Lipschitz continuous endomorphism
of $H_\alpha$. More precisely, one has
\[
  \norm[\big]{hIh - gIg}_{H_\alpha} \lesssim \bigl(
  \norm{h}_{H_\alpha} + \norm{g}_{H_\alpha} \bigr) \norm{h-g}_{H_\alpha}
  \qquad \forall h,g \in H_\alpha.
\]
The following simple extension will be used several times.
\begin{lemma}
  \label{lm:hagi}
  The map formally defined on $(H_\alpha)^\enne$ as
  \[
  h \longmapsto \ip{h}{Ih} := \sum_{k=1}^\infty h_kIh_k
  \]
  is a locally Lipschitz continuous map from $\ell^2(H_\alpha)$ to
  $H_\alpha$. More precisely,
  \[
  \norm[\big]{\ip{h}{Ih} - \ip{g}{Ig}}_{H_\alpha} \lesssim
  \bigl(\norm{h}_{\ell^2(H_\alpha)} + \norm{g}_{\ell^2(H_\alpha)}
  \bigr) \norm{h-g}_{\ell^2(H_\alpha)}
  \]
  for every $h, g \in \ell^2(H_\alpha)$.
\end{lemma}
\begin{proof}
  Let $h$, $g \in \ell^2(H_\alpha)$. The Minkowski and Cauchy-Schwarz
  inequalities imply
  \begin{align*}
    \norm[\big]{\ip{h}{Ih} - \ip{g}{Ig}}_{H_\alpha}
    &= \norm[\Big]{\sum_{k=1}^\infty \bigl( h_kIh_k - g_kIg_k \bigr)}_{H_\alpha}\\
    &\leq \sum_{k=1}^\infty \norm[\big]{h_kIh_k - g_kIg_k}_{H_\alpha}\\
    &\lesssim \sum_{k=1}^\infty \bigl(\norm{h_k}_{H_\alpha} +
    \norm{g_k}_{H_\alpha}
    \bigr) \norm{h_k-g_k}_{H_\alpha}\\
    &\leq \bigl(\norm{h}_{\ell^2(H_\alpha)} +
    \norm{g}_{\ell^2(H_\alpha)} \bigr) \norm{h-g}_{\ell^2(H_\alpha)}.
  \qedhere
  \end{align*}
\end{proof}

As an immediate corollary it follows that
\[
  \norm[\big]{\beta(v_1) - \beta(v_2)}_{H_\alpha}
  \lesssim \bigl( \norm[\big]{\sigma(v_1)}_{\cL^2(U,H_\alpha)}
  + \norm[\big]{\sigma(v_2)}_{\cL^2(U,H_\alpha)} \bigr)
  \norm[\big]{\sigma(v_1) - \sigma(v_2)}_{\cL^2(U,H_\alpha)},
\]
for any $v_1$, $v_2 \in H_\alpha$ (here and in the following we
suppress the explicit indication of the dependence on $\omega$ and
$t$). In particular, $\beta$ is Lipschitz continuous if $\sigma$ is
Lipschitz continuous and bounded, uniformly over
$\Omega \times \erre_+$. If $\sigma$ is just locally Lipschitz
continuous and locally bounded, uniformly over
$\Omega \times \erre_+$, the same holds for $\beta$.

It turns out, however, that diffusion coefficients $\sigma=(\sigma_k)$
given by superposition operators are \emph{not} Lipschitz continuous
and bounded, even for very regular functions, so that global
well-posedness (in the mild sense) of \eqref{eq:ugo} is not
guaranteed. However, we are going to show that, under suitable
conditions, they are locally Lipschitz continuous and bounded, so that
\eqref{eq:ugo} is locally well posed. Analogous results are proved in
\cite[{\S}5.4]{filipo}, but we provide nonetheless a proof for several
reasons: we use a different norm on $H_\alpha$, our assumptions are
slightly different, and we shall extensively employ these estimates
later.

In the following, for any function
$\phi: \Omega \times \erre_+ \times \erre_+ \times \erre \to \erre$,
we shall denote by $\partial_1\phi$ and $\partial_2\phi$ the partial
derivatives of $\phi$ with respect to its third and fourth argument,
respectively.

\begin{hyp}
  Let $\psi = (\psi_k) \in \ell^2(L^2_{\alpha})$, with $\psi_k \geq 0$
  for all $k \in \enne$, and $\eta = (\eta_k) \subset \erre^\erre$ be
  a sequence of positive increasing even functions such that
  $\widetilde{\eta} := \norm{\eta(\cdot)}_{\ell^2}: \erre \to \erre$ is
  bounded on bounded sets. The functions
  \[
  \sigma_k: \Omega \times \erre_+^2 \times \erre \to \erre, \qquad k
  \in \enne,
  \]
  are measurable with respect to the $\sigma$-algebra
  $\mathscr{P} \otimes \mathscr{B}(\erre_+) \otimes
  \mathscr{B}(\erre)$, and
  $\sigma_k(\omega,t,\cdot,\cdot) \in C^1(\erre \times \erre_+)$ for
  all $(\omega,t) \in \Omega \times \erre_+$. Moreover, they
  satisfy the following conditions:
  \begin{itemize}
  \item[(a)] $\lim_{x\to\infty} \sigma_k(\omega,t,x,r)=0$ for all $r \in \erre$;
  \item[(b)] $\abs[\big]{\partial_1 \sigma_k(\omega,t,x,r)} \leq
    \psi_k(x) \eta_k(r)$ for all $(\omega,t,x,r) \in \Omega \times
    \erre_+^2 \times \erre$;
  \item[(c)] $\abs[\big]{\partial_2 \sigma_k(\omega,t,x,r)} \leq
    \eta_k(r)$ for all $(\omega,t,x,r) \in \Omega \times \erre_+^2
    \times \erre$;
  \item[(d)] $\abs[\big]{\partial_1 \sigma_k(\omega,t,x,r_1)
      - \partial_1 \sigma_k(\omega,t,x,r_2)} \leq \psi_k(x)
    \bigl(\eta_k(r_1) + \eta_k(r_2)\bigr) \abs{r_1-r_2}$ for all
    $(\omega,t,x) \in \Omega \times \erre_+^2$ and $r_1, r_2 \in
    \erre$;
  \item[(e)] $\abs[\big]{\partial_2 \sigma_k(\omega,t,x,r_1)
      - \partial_2 \sigma_k(\omega,t,x,r_2)} \leq \bigl(\eta_k(r_1) +
    \eta_k(r_2)\bigr) \abs{r_1-r_2}$ for all $(\omega,t,x) \in \Omega
    \times \erre_+^2$ and $r_1, r_2 \in \erre$.
  \end{itemize}
\end{hyp}

\begin{prop}
  \label{prop:pippo}
  Assume that Hypothesis~1 is satisfied.  Then $\sigma$ is a
  well-defined map from $\Omega \times \erre_+ \times H_\alpha$ to
  $\cL^2(U,H_\alpha)$, measurable with respect to the $\sigma$-algebra
  $\mathscr{P} \otimes \mathscr{B}(H_\alpha)$, and it satisfies the
  estimates
  \begin{align*}
    \norm[\big]{\sigma(\omega,t,v)}_{\cL^2(U,H_\alpha)}
    &\leq \widetilde{\eta}\bigl( \delta \norm{v}_{H_\alpha} \bigr)
      \bigl( \norm{v}_{H_\alpha} + \norm{\psi}_{\ell^2(L^2_\alpha)} \bigr),\\
    \norm[\big]{\sigma(\omega,t,v_1)-\sigma(\omega,t,v_2)}_{\cL^2(U,H_\alpha)}
    &\lesssim_\alpha
    \Bigl( \widetilde{\eta}\bigl(%
    \delta \norm{v_1}_{H_\alpha} \bigr) + \widetilde{\eta}\bigl(%
      \delta \norm{v_2}_{H_\alpha} \bigr) \Bigr) \, \cdot\\
    &\qquad \cdot \, \bigl(%
    \norm{v_2}_{H_\alpha} + \norm{\psi}_{\ell^2(L^2_\alpha)} \bigr)
    \norm[\big]{v_1-v_2}_{H_\alpha}
  \end{align*}
  for all $(\omega,t) \in \Omega \times \erre_+$ and
  $v, v_1, v_2 \in H_\alpha$. In particular, $\sigma$ is locally
  bounded and locally Lipschitz continuous in its third argument,
  uniformly over $\Omega \times \erre_+$.
\end{prop}
\begin{proof}
  Throughout the proof we shall omit the explicit indication of the
  first two arguments of $\sigma$ as well as of $\sigma_k$. Since
  $\sigma_k(\cdot,r)$ is zero at infinity for all $r \in \erre$, one
  has, by the triangle inequality in $\ell^2(L^2_{\alpha})$,
  \begin{align*}
  \norm[\big]{\sigma(v)}_{\cL^2(U,H_\alpha)}
  &= \biggl( \sum_{k=1}^\infty
  \norm[\big]{\partial_1 \sigma_k(\cdot,v)
  + \partial_2 \sigma_k(\cdot,v)v'}^2_{L^2_{\alpha}} \biggr)^{1/2}\\
  &\leq \biggl( \sum_{k=1}^\infty \norm[\big]{\partial_1%
  \sigma_k(\cdot,v)}^2_{L^2_{\alpha}} \biggr)^{1/2}
  + \biggl( \sum_{k=1}^\infty \norm[\big]{\partial_2%
  \sigma_k(\cdot,v)v'}^2_{L^2_{\alpha}} \biggr)^{1/2}
  \end{align*}
  for every $v \in H_\alpha$. It follows by (b) that
  \[
  \norm[\big]{\partial_1 \sigma_k(\cdot,v)}_{L^2_{\alpha}} \leq
  \norm[\big]{\psi_k \eta_k(v)}_{L^2_{\alpha}} \leq
  \norm[\big]{\psi_k}_{L^2_{\alpha}} \,
  \norm[\big]{\eta_k(v)}_{L^\infty},
  \]
  where, denoting by $\delta=\delta(\alpha)$ the (operator) norm of
  the embedding $H_\alpha \embed L^\infty$, the assumptions on
  $\eta_k$ imply
  \[
  \norm[\big]{\eta_k(v)}_{L^\infty} \leq \eta_k\bigl( \norm{v}_{L^\infty} \bigr)
  \leq \eta_k\bigl( \delta \norm{v}_{H_\alpha} \bigr),
  \]
  hence also
  \[
  \norm[\big]{\partial_1 \sigma_k(\cdot,v)}_{L^2_{\alpha}}
  \leq \norm[\big]{\psi_k}_{L^2_{\alpha}} \,
  \eta_k\bigl( \delta \norm{v}_{H_\alpha} \bigr).
  \]
  Therefore, applying the Cauchy-Schwarz inequality and
  $\norm{\cdot}_{\ell^4} \leq \norm{\cdot}_{\ell^2}$, one has
  \begin{align*}
    \biggl( \sum_{k=1}^\infty \norm[\big]{\partial_1%
    \sigma_k(\cdot,v)}^2_{L^2_{\alpha}} \biggr)^{1/2}
    &\leq \biggl( \sum_{k=1}^\infty \norm[\big]{\psi_k}^2_{L^2_{\alpha}}
      \eta^2_k\bigl( \delta \norm{v}_{H_\alpha} \bigr) \biggr)^{1/2}\\
    &\leq \biggl( \sum_{k=1}^\infty \norm[\big]{\psi_k}^4_{L^2_{\alpha}}
      \biggr)^{1/4}%
      \biggl( \sum_{k=1}^\infty \eta^4_k\bigl( \delta \norm{v}_{H_\alpha} \bigr)
      \biggr)^{1/4}\\
    &\leq \norm[\big]{\psi}_{\ell^2(L^2_\alpha)} \, \widetilde{\eta}\bigl(
      \delta \norm{v}_{H_\alpha} \bigr).
  \end{align*}
  Similarly, (c) yields
  \begin{equation*}
    \norm[\big]{\partial_2 \sigma_k(\cdot,v)v'}_{L^2_{\alpha}}
    \leq \norm[\big]{\eta_k(v) v'}_{L^2_{\alpha}}
    \leq \norm[\big]{\eta_k(v)}_{L^\infty} \norm[\big]{v}_{H_\alpha}
    \leq \eta_k\bigl( \delta \norm{v}_{H_\alpha} \bigr)
    \norm[\big]{v}_{H_\alpha},
  \end{equation*}
  hence
  \begin{align*}
    \biggl( \sum_{k=1}^\infty \norm[\big]{\partial_2%
    \sigma_k(\cdot,v)v'}^2_{L^2_{\alpha}} \biggr)^{1/2}
    &\leq \widetilde{\eta}\bigl( \delta \norm{v}_{H_\alpha} \bigr)
    \norm[\big]{v}_{H_\alpha}.
  \end{align*}
  We have thus shown that
  \[
  \norm[\big]{\sigma(v)}_{\cL^2(U,H_\alpha)} \leq
  \widetilde{\eta}\bigl( \delta \norm{v}_{H_\alpha} \bigr)
  \bigl( \norm{v}_{H_\alpha} + \norm{\psi}_{\ell^2(L^2_\alpha)} \bigr),
  \]
  from which it follows immediately that $\sigma$ is well defined and
  locally bounded, recalling that $\widetilde{\eta}$ is bounded on
  bounded sets.
  Let us now establish the local Lipschitz continuity of $\sigma$. In
  analogy to a previous computation, one has, for any $v_1, v_2 \in
  H_\alpha$,
  \begin{align*}
    \norm[\big]{\sigma(v_1)-\sigma(v_2)}_{\cL^2(U,H_\alpha)} &\leq
    \biggl( \sum_{k=1}^\infty \norm[\big]{\partial_1 \sigma_k(\cdot,v_1)
    - \partial_1\sigma_k(\cdot,v_2)}^2_{L^2_{\alpha}} \biggr)^{1/2}\\
    &\quad + \biggl( \sum_{k=1}^\infty \norm[\big]{\partial_2%
    \sigma_k(\cdot,v_1)v_1'
    - \partial_2\sigma_k(\cdot,v_2)v_2'}^2_{L^2_{\alpha}}
    \biggr)^{1/2},
  \end{align*}
  where, thanks to (d),
  \begin{align*}
  \norm[\big]{\partial_1 \sigma_k(\cdot,v_1)
    - \partial_1\sigma_k(\cdot,v_2)}^2_{L^2_{\alpha}}
  &\leq \norm[\big]{\psi_k (\eta_k(v_1) + \eta_k(v_2))
    \abs{v_1-v_2}}_{L^2_{\alpha}}\\
  &\leq \norm[\big]{\psi_k}_{L^2_{\alpha}}
        \norm[\big]{(\eta_k(v_1) + \eta_k(v_2))}_{L^\infty}
        \norm[\big]{v_1 - v_2}_{L^\infty}\\
  &\leq \delta \norm[\big]{\psi_k}_{L^2_{\alpha}} \,
  \bigl( \eta_k(\delta \norm{v_1}_{H_\alpha})
  + \eta_k(\delta \norm{v_2}_{H_\alpha}) \bigr)
  \norm[\big]{v_1-v_2}_{H_\alpha},
  \end{align*}
  which implies, similarly to a previous computation,
  \begin{align*}
    &\biggl( \sum_{k=1}^\infty \norm[\big]{\partial_1 \sigma_k(\cdot,v_1)
    - \partial_1\sigma_k(\cdot,v_2)}^2_{L^2_{\alpha}} \biggr)^{1/2}\\
    &\hspace{3em} \leq \delta \norm[\big]{\psi}_{\ell^2(L^2_{\alpha})} \,
  \Bigl( \widetilde{\eta}\bigl(\delta \norm{v_1}_{H_\alpha}\bigr)
  + \widetilde{\eta}\bigl(\delta \norm{v_2}_{H_\alpha}\bigr) \Bigr)
  \norm[\big]{v_1-v_2}_{H_\alpha}.
  \end{align*}
  Moreover, one has
  \begin{align*}
  &\norm[\big]{\partial_2 \sigma_k(\cdot,v_1)v_1'
    - \partial_2\sigma_k(\cdot,v_2)v_2'}_{L^2_{\alpha}}\\
  &\hspace{3em} \leq \norm[\big]{\partial_2 \sigma_k(\cdot,v_1)%
    (v_1' - v_2')}_{L^2_{\alpha}}
  + \norm[\big]{\bigl(\partial_2 \sigma_k(\cdot,v_1) 
    - \partial_2\sigma_k(\cdot,v_2)\bigr)v_2'}_{L^2_{\alpha}},
  \end{align*}
  where, by (c),
  \begin{align*}
  \norm[\big]{\partial_2 \sigma_k(\cdot,v_1) (v_1' - v_2')}^2_{L^2_{\alpha}}
  &\leq \norm[\big]{\eta_k(v_1)}_{L^\infty} \norm[\big]{v_1-v_2}_{H_\alpha}\\
  &\leq \eta_k\bigl( \delta \norm{v_1}_{H_\alpha} \bigr)
  \norm[\big]{v_1-v_2}_{H_\alpha},
  \end{align*}
  and, by (e),
  \begin{align*}
  &\norm[\big]{\bigl(\partial_2 \sigma_k(\cdot,v_1)
    - \partial_2\sigma_k(\cdot,v_2)\bigr)v_2'}_{L^2_{\alpha}}\\
  &\hspace{3em} \leq \norm[\big]{\bigl(\eta_k(v_1) +
    \eta_k(v_2)\bigr) (v_1-v_2) v_2'}_{L^2_{\alpha}}\\
  &\hspace{3em} \leq \norm[\big]{\eta_k(v_1) +
    \eta_k(v_2)}_{L^\infty} \, \norm[\big]{v_1-v_2}_{L^\infty}
  \, \norm[\big]{v_2}_{H_\alpha}\\
  &\hspace{3em} \leq \delta \Bigl( \eta_k\bigl(%
   \delta \norm{v_1}_{H_\alpha} \bigr) + \eta_k\bigl(%
   \delta \norm{v_2}_{H_\alpha} \bigr) \Bigr)
   \norm[\big]{v_2}_{H_\alpha} \, \norm[\big]{v_1-v_2}_{H_\alpha},
  \end{align*}
  so that
  \begin{align*}
  &\norm[\big]{\partial_2 \sigma_k(\cdot,v_1)v_1'
    - \partial_2\sigma_k(\cdot,v_2)v_2'}_{L^2_{\alpha}}\\
  &\hspace{3em} \lesssim_\alpha \Bigl( \eta_k\bigl(%
   \delta \norm{v_1}_{H_\alpha} \bigr) + \eta_k\bigl(%
   \delta \norm{v_2}_{H_\alpha} \bigr) \Bigr)
   \norm[\big]{v_2}_{H_\alpha} \, \norm[\big]{v_1-v_2}_{H_\alpha}
  \end{align*}
  and
  \begin{align*}
    &\biggl( \sum_{k=1}^\infty \norm[\big]{\partial_2 \sigma_k(\cdot,v_1)v_1'
    - \partial_2\sigma_k(\cdot,v_2)v_2'}_{L^2_{\alpha}} \biggr)^{1/2}\\
    &\hspace{3em} \lesssim_\alpha \Bigl( \widetilde{\eta}\bigl(%
   \delta \norm{v_1}_{H_\alpha} \bigr) + \widetilde{\eta}\bigl(%
   \delta \norm{v_2}_{H_\alpha} \bigr) \Bigr)
   \norm[\big]{v_2}_{H_\alpha} \, \norm[\big]{v_1-v_2}_{H_\alpha}.
  \end{align*}
  We have thus proved that
  \begin{align*}
    &\norm[\big]{\sigma(v_1)-\sigma(v_2)}_{\cL^2(U,H_\alpha)}\\
    &\hspace{3em} \lesssim_\alpha
    \Bigl( \widetilde{\eta}\bigl(%
    \delta \norm{v_1}_{H_\alpha} \bigr) + \widetilde{\eta}\bigl(%
    \delta \norm{v_2}_{H_\alpha} \bigr) \Bigr) \Bigl(%
    \norm[\big]{\psi}_{\ell^2(L^2_\alpha)} + \norm[\big]{v_2}_{H_\alpha} \Bigr)
    \norm[\big]{v_1-v_2}_{H_\alpha},
  \end{align*}
  which implies, thanks to the assumptions on $\eta$, the asserted
  local Lipschitz continuity of $\sigma$.
\end{proof}

\subsection{Well-posedness on $L^2_{-\alpha}$}
Let us begin with the following estimate.
\begin{lemma}
  Let $\alpha>0$, and $f=(f_k)$, $g=(g_k)$ be sequences of real-valued
  functions on $\erre_+$. Then
  \[
    \norm[\big]{\ip{f}{Ig}_{\ell^2}}_{L^2_{-\alpha}}
    \lesssim_\alpha \norm[\big]{f}_{\ell^2(L^2_\alpha)}^2 \;
      \norm[\big]{g}^2_{\ell^2(L^2_{-\alpha})}.
  \]
\end{lemma}
\begin{proof}
  One has
  \begin{align*}
    \norm[\big]{\ip{f}{Ig}_{\ell^2}}^2_{L^2_{-\alpha}}
    &= \int_0^\infty \abs[\bigg]{\sum_k f_k(x) \int_0^x g_k(y)\,dy}^2
      e^{-\alpha x}\,dx\\
    &\leq \int_0^\infty \abs[\bigg]{\sum_k f_k(x) e^{-\frac{\alpha}{2}x}
      \int_0^x g_k(y)\,dy}^2 \,dx\\
    &\leq \int_0^\infty \abs[\bigg]{\sum_k f_k(x) e^{\frac{\alpha}{2}x}
      \int_0^\infty \abs{g_k(y)} e^{-\alpha y}\,dy}^2 \,dx,
  \end{align*}
  where
  \[
    \int_0^\infty \abs{g_k(y)} e^{-\alpha y}\,dy
    = \norm[\big]{g_k}_{L^1_{-\alpha}}
    \lesssim_\alpha \norm[\big]{g_k}_{L^2_{-\alpha}},
  \]
  thus also 
  \begin{align*}
    \sum_k f_k(x) e^{\frac{\alpha}{2}x} \int_0^\infty \abs{g_k(y)}
    e^{-\alpha y}\,dy
    &= \ip[\big]{(f_k e^{\frac{\alpha}{2}\cdot})}%
      {\norm{g_k}_{L^2_{-\alpha}}}_{\ell^2}\\
    &\leq \norm[\big]{(f_k e^{\frac{\alpha}{2}\cdot})}_{\ell^2}
      \norm[\big]{g}_{\ell^2(L^2_{-\alpha})}.
  \end{align*}
  Therefore
  \begin{align*}
    \norm[\big]{\ip{f}{Ig}_{\ell^2}}^2_{L^2_{-\alpha}}
    &\lesssim_\alpha \norm[\big]{\norm{(f_k)}_{\ell^2}%
      e^{\frac{\alpha}{2}\cdot}}^2_{L^2} \;
      \norm[\big]{g}^2_{\ell^2(L^2_{-\alpha})}\\
    &= \norm[\big]{f}_{L^2_\alpha(\ell^2)}^2 \;
      \norm[\big]{g}^2_{\ell^2(L^2_{-\alpha})}.
      \qedhere
  \end{align*}
\end{proof}

We can now provide conditions on $\sigma$ implying that $\beta$
satisfies a suitable form of Lipschitz continuity.
\begin{lemma}
  \label{lm:lipL}
  Assume that $\sigma(\omega,t,\cdot)$ is Lipschitz continuous and
  bounded from $L^2_{-\alpha}$ to $\cL^2(U,L^2_\alpha)$ uniformly
  with respect to $(\omega,t) \in \Omega \times \erre_+$, i.e. that
  there exists a constant $N$ such that
  \[
    \sum_{k=1}^\infty
    \norm[\big]{\sigma_k(v_1) - \sigma_k(v_2)}^2_{L^2_{-\alpha}}
    \leq N \norm[\big]{v_1-v_2}^2_{L^2_{-\alpha}}, \qquad
    \sum_{k=1}^\infty \norm[\big]{\sigma_k(v)}^2_{L^2_\alpha}
    < N
  \]
  for all $v$, $v_1$, $v_2 \in L^2_{-\alpha}$, uniformly over
  $\Omega \times \erre_+$. Then $\beta(\omega,t,\cdot)$ is a Lipschitz
  continuous endomorphism of $L^2_{-\alpha}$, uniformly with respect
  to $(\omega,t) \in \Omega \times \erre_+$.
\end{lemma}
\begin{proof}
  For any $v_1$, $v_2 \in L^2_{-\alpha}$ one has
  \begin{align*}
    \norm[\big]{\beta(v_1) - \beta(v_2)}_{L^2_{-\alpha}}
    &= \norm[\big]{\ip[\big]{\sigma(v_1)}{I\sigma(v_1)}_{\ell^2}
      - \ip[\big]{\sigma(v_2)}{I\sigma(v_2)}_{\ell^2}}_{L^2_{-\alpha}}\\
    &= \norm[\big]{\ip[\big]{\sigma(v_1)-\sigma(v_2)}%
      {I\sigma(v_1)}_{\ell^2} - \ip[\big]{\sigma(v_2)}%
      {I\sigma(v_1) - I\sigma(v_2)}_{\ell^2}}_{L^2_{-\alpha}}\\
    &= \norm[\big]{\ip[\big]{\sigma(v_1)%
      -\sigma(v_2)}{I\sigma(v_1)}_{\ell^2}}_{L^2_{-\alpha}}
      + \norm[\big]{\ip[\big]{\sigma(v_2)}%
      {I\sigma(v_1) - I\sigma(v_2)}_{\ell^2}}_{L^2_{-\alpha}},
  \end{align*}
  where
  \begin{align*}
    &\norm[\big]{\ip[\big]{\sigma(v_1)%
    -\sigma(v_2)}{I\sigma(v_1)}_{\ell^2}}_{L^2_{-\alpha}}\\
    &\hspace{3em} \leq \norm[\big]{\norm{\sigma(v_1) - \sigma(v_2)}_{\ell^2}
      \norm{I\sigma(v_1)}_{\ell^2}}_{L^2_{-\alpha}}\\
    &\hspace{3em} \leq \norm[\big]{\sigma(v_1) %
      - \sigma(v_2)}_{\cL^2(\ell^2,L^2_{-\alpha})} \;
      \norm[\big]{I\sigma(v_1)}_{L^\infty(\ell^2)},
  \end{align*}
  and
  \[
    \norm[\big]{I\sigma(v_1)}_{\ell^2}
    = \Bigl( \sum_k \abs[\big]{I\sigma_k(v_1)}^2 \Bigr)^{1/2}
    \leq \Bigl( \sum_k \norm[\big]{\sigma_k(v_1)}_{L^1}^2 \Bigr)^{1/2}
    = \norm[\big]{\sigma(v_1)}_{\ell^2(L^1)},
  \]
  hence, recalling that
  $L^2_\alpha \embed L^1$, 
  \[
    \norm[\big]{\ip[\big]{\sigma(v_1)%
        -\sigma(v_2)}{I\sigma(v_1)}_{\ell^2}}_{L^2_{-\alpha}}
    \leq N \, \norm[\big]{\sigma(v_1)}_{\ell^2(L^2_\alpha)} \;
    \norm[\big]{v_1-v_2}_{L^2_{-\alpha}}.
  \]
  Furthermore, the previous lemma yields, by linearity of $I$, 
  \begin{align*}
    \norm[\big]{\ip[\big]{\sigma(v_2)}%
    {I\sigma(v_1) - I\sigma(v_2)}_{\ell^2}}_{L^2_{-\alpha}}
    &\lesssim_\alpha \norm[\big]{\sigma(v_2)}_{\ell^2(L^2_\alpha)} \;
      \norm[\big]{\sigma(v_1)-\sigma(v_2)}_{\ell^2(L^2_{-\alpha})}\\
    &\leq N \, \norm[\big]{\sigma(v_2)}_{\ell^2(L^2_\alpha)} \;
      \norm[\big]{v_1-v_2}_{L^2_{-\alpha}}.
  \end{align*}
  Recalling that $\sigma$ is bounded from $L^2_{-\alpha}$ to
  $\cL^2(\ell^2,L^2_\alpha)$, we conclude that
  \[
    \norm[\big]{\beta(v_1) - \beta(v_2)}_{L^2_{-\alpha}}
    \lesssim \norm[\big]{v_1-v_2}_{L^2_{-\alpha}},
  \]
  where the implicit constant depends on $\alpha$
  and $N$.
\end{proof}

We can now give a well-posedness result for Musiela's equation on the
state space $L^2_{-\alpha}$.
\begin{prop}
  \label{prop:WP}
  Let $p>0$. Assume that $\sigma$ satisfies the assumptions of
  lemma~\ref{lm:lipL} and that
  $u_0 \in L^p(\Omega;\cF_0;L^2_{-\alpha})$. Then equation
  \eqref{eq:ugo} admits a unique mild solution
  $u \in L^p(\Omega;C([0,T];L^2_{-\alpha}))$ for every
  $T>0$. Moreover, the solution $u$ depends continuously on the
  initial datum $u_0$. More precisely, the map
  \begin{align*}
    L^p(\Omega;\cF_0;L^2_{-\alpha})
    &\longrightarrow L^p(\Omega;C([0,T];L^2_{-\alpha}))\\
    u_0 &\longmapsto u
  \end{align*}
  is Lipschitz continuous.  
\end{prop}
The proof follows by general well-posedness results in the mild sense
for stochastic evolution equations in Hilbert spaces with Lipschitz
continuous nonlinearities (see, e.g., \cite{DPZ}).

\medskip

As in the previous subsection, we shall now focus on the case where,
for every $k \in \enne$, $\sigma_k$ is given by a (random)
superposition operator acting on $u$: with a harmless abuse (or,
better, overload) of notation, we assume that
\[
  \sigma_k(\omega,t,u) = \sigma_k(\omega,t,\cdot,u(t,\cdot)),
\]
where the function
$\sigma_k: \Omega \times \erre_+ \times \erre_+ \times \erre \to
\erre$ is measurable with respect to the $\sigma$-algebra
$\mathscr{P} \otimes \mathscr{B}(\erre_+) \otimes \mathscr{B}(\erre)$.

Existence and uniqueness of a mild solution to Musiela's equation
is established next. 
\begin{prop}
  \label{prop:oli}
  Assume that the functions
  \[
  \sigma_k: \Omega \times \erre_+^2 \times \erre \to \erre, \qquad k
  \in \enne,
  \]
  satisfy the following conditions:
  \begin{itemize}
  \item[\emph{(a)}] there exists $\theta = (\theta_k) \in \ell^2(L^2_\alpha)$
    such that $\abs[\big]{\sigma_k(\omega,t,x,r)} \leq \theta_k(x)$ for
    all $(\omega,t,x,r) \in \Omega \times \erre^2_+ \times \erre$.
  \item[\emph{(b)}] there exists $c=(c_k) \in \ell^2$ such that
    $\abs[\big]{\sigma_k(\omega,t,x,r_1) - \sigma_k(\omega,t,x,r_1)}
    \leq c_k \abs{r_1-r_2}$ for all
    $(\omega,t,x) \in \Omega \times \erre^2_+$ and $r_1$,
    $r_2 \in \erre$.
  \end{itemize}
  If $u_0 \in L^p(\Omega;\cF_0;L^2_{-\alpha})$, $p>0$, then all
  assertions of proposition~\ref{prop:WP} hold.
\end{prop}
\begin{proof}
  It suffices to check that (a) and (b) imply that $\sigma$ satisfies
  the assumptions of lemma~\ref{lm:lipL}. In fact, by (a) it follows
  that
  \begin{align*}
  \norm[\big]{\sigma(v_1) - \sigma(v_2)}_{\cL^2(U,L^2_{-\alpha})}
  &= \Biggl( \sum_k \norm[\big]{\sigma_k(v_1)-\sigma_k(v_2)}^2_{L^2_{-\alpha}}
    \Biggr)^{1/2}\\
  &= \Biggl( \sum_k c_k^2 \Biggr)^{1/2}
    \norm[\big]{v_1-v_2}_{L^2_{-\alpha}}
  \end{align*}
  for all $v_1$, $v_2 \in L^2_{-\alpha}$,
  i.e. $\sigma(\omega,t,\cdot)$ is Lipschitz continuous from
  $L^2_{-\alpha}$ to $\cL^2(\ell^2,L^2_{-\alpha})$, uniformly with
  respect to its other arguments, with Lipschitz constant
  $\norm{c}_{\ell^2}$.
  Moreover, (b) implies
  \[
    \norm[\big]{\sigma(v)}^2_{\cL^2(U,L^2_\alpha)}
    = \sum_{k=1}^\infty \norm[\big]{\sigma_k(v)}^2_{L^2_\alpha}
    \leq \sum_{k=1}^\infty \norm[\big]{\theta_k}^2_{L^2_\alpha}
    = \norm[\big]{\theta}^2_{\ell^2(L^2_\alpha)} < \infty
  \]
  for all $v \in L^2_{-\alpha}$ and
  $(\omega,t) \in \Omega \times \erre_+$. We have hence shown that the
  assumptions of lemma~\ref{lm:lipL} are satisfied.
\end{proof}

\begin{rmk}
  Let $\alpha>0$. It is proved in \cite[{\S}5.5]{filipo} that
  Musiela's equation is globally well posed on the state space
  $H_\alpha$, rather than just locally, if, for every $k \in \enne$,
  $\ip{\sigma(v)}{e_k}(x) = \sigma_k(x,\zeta_k(v))$, where
  $\zeta_k \in \cL(H_\alpha,\erre)$, and
  \begin{itemize}
  \item[(a)] $\abs[\big]{\partial_1 \sigma_k(x,r)} \leq \psi_k(x)$;
  \item[(b)]
    $\abs[\big]{\partial_1 \sigma_k(x,r_1) - \partial_1
      \sigma_k(x,r_2)} \leq \psi_k(x)\abs{r_1-r_2}$.
  \end{itemize}
  Here $\psi=(\psi_k)$ and the measurability assumptions on the
  functions $\sigma_k$ are the same as in hypothesis~1. It is natural
  to ask whether the conditions (a) and (b) imply well-posedness on
  the state space $L^2_{-\alpha}$. We are going to show that a bit
  more integrability on $\psi$ implies that this indeed the case.  We
  point out, however, that (a) and (b) alone imply only the existence
  of a local mild solution on the state space $H_\alpha$. Since
  $H_\alpha \embed L^2_{-\alpha}$, the $H_\alpha$-valued solution
  coincides with the $L^2_{-\alpha}$-valued solution on the lifetime of
  the former.

  Let us show that (a) and (b) above, with
  $\psi \in L^2_{\alpha+\varepsilon}(\ell^2)$, $\varepsilon>0$, imply
  that $\sigma$ satisfies the assumptions of Lemma~\ref{lm:lipL}. In
  fact, the identity
  \[
  \sigma_k(\infty,r) - \sigma_k(x,r) = - \sigma_k(x,r) =
  \int_x^\infty \partial_1 \sigma_k(y,r)\,dy
  \]
  implies
  \begin{align*}
    \norm[\big]{\sigma(v)}^2_{\cL^2(U,L^2_\alpha)}
    &= \sum_{k=1}^\infty \norm[\big]{\sigma_k(v)}^2_{L^2_\alpha}\\
    &\leq \sum_{k=1}^\infty \int_0^\infty \abs[\bigg]{%
      \int_x^\infty \psi_k(y)\,dy}^2 e^{\alpha x} \,dx\\
    &\leq \sum_{k=1}^\infty \int_0^\infty \abs[\bigg]{%
      \int_x^\infty \psi_k(y) e^{\frac{\alpha+\varepsilon}{2}y}
      e^{-\frac{\varepsilon}{2}y} \,dy}^2 \,dx,
  \end{align*}
  where, by the Cauchy-Schwarz inequality,
  \[
  \int_x^\infty \psi_k(y) e^{\frac{\alpha+\varepsilon}{2}y}
  e^{-\frac{\varepsilon}{2}y} \,dy \leq
  \norm[\big]{\psi_k}_{L^2_{\alpha+\varepsilon}} \biggl(
  \int_x^\infty e^{-\varepsilon y} \,dy \biggr)^{1/2} =
  \frac{1}{\sqrt{\varepsilon}} e^{-\frac{\varepsilon}{2}x}
  \norm[\big]{\psi_k}_{L^2_{\alpha+\varepsilon}},
  \]
  hence
  \[
  \norm[\big]{\sigma(v)}^2_{\cL^2(U,L^2_\alpha)} \leq
  \frac{1}{\varepsilon} \sum_{k=1}^\infty
  \norm[\big]{\psi_k}^2_{L^2_{\alpha+\varepsilon}} \int_0^\infty
  e^{-\varepsilon x} \,dx = \frac{1}{\varepsilon^2}
  \norm[\big]{\psi}^2_{L^2_{\alpha+\varepsilon}(\ell^2)}.
  \]
  Let us now consider Lipschitz continuity: from
  \begin{align*}
    &\sigma_k(\infty,r_1) - \sigma_k(\infty,r_2)
    - \bigl( \sigma_k(x,r_1) - \sigma_k(x,r_2) \bigr)\\
    &\hspace{3em} = - \bigl( \sigma_k(x,r_1) - \sigma_k(x,r_2) \bigr)
    = \int_x^\infty \bigl( \partial_1 \sigma_k(y,r_1) - \partial_1
    \sigma_k(y,r_2) \bigr) \,dy
  \end{align*}
  and condition (b) above it follows that
  \begin{align*}
    \abs[\big]{\sigma_k(x,r_1) - \sigma_k(x,r_2)} 
    &\leq \abs{r_1-r_2}
    \int_x^\infty \abs{\psi_k(y)} \,dy \leq \abs{r_1-r_2}
    \int_0^\infty \abs{\psi_k(y)} e^{\frac{\alpha}{2}y}
    e^{-\frac{\alpha}{2}y} \,dy\\
    &\lesssim_\alpha \norm[\big]{\psi_k}_{L^2_\alpha}
    \abs{r_1-r_2}.
  \end{align*}
  Since the $\ell^2$-norm of $\norm[\big]{\psi_k}_{L^2_\alpha}$ is
  finite by assumption, $\sigma_k$ satisfies hypothesis~2(a) with
  $c_k:=\norm[\big]{\psi_k}_{L^2_\alpha}$, which in turn implies, as
  in the proof of proposition~\ref{prop:oli}, the Lipschitz continuity
  of $\sigma$ asserted in lemma~\ref{lm:lipL}.
\end{rmk}

\section{Approximations I}
\label{sec:approx1}
Here we consider Musiela's equation with $\ip{\sigma}{e_k}$ given by
superposition operators associated to functions
\[
  \sigma_k: \Omega \times \erre_+^2 \times \erre \to \erre, \qquad
  k \in \enne,
\]
we introduce approximations of such functions, and we prove several
estimates and convergence results for the solutions to the
approximated equations.

For each $n \in \enne$, let $\chi_n \in C^\infty_c(\erre_+)$ be a
smooth cut-off function, i.e. $\chi_n=1$ on $[0,n]$,
$0 \leq \chi_n \leq 1$ and $\abs{\chi'_n} \leq 2$ in
$\mathopen]n,n+1\mathclose[$, and $\chi_n=0$ on
$[n+1,\infty\mathclose[$. For every $k$ and $n \in \enne$, let us
set
\[
  \sigma_k^{(n)}(x,r) := \sigma_k(x,r) \chi_n(x) \qquad \forall (x,r)
  \in \erre_+ \times \erre.
\]
Here and in the following, as already done before, we suppress the
explicit indication of the dependence on
$(\omega,t) \in \Omega \times \erre_+$ where it is not essential.

The next lemma, while rather simple, is of crucial importance.
\begin{lemma}
  \label{lm:pepsi}
  Let $\psi=(\psi_k) \in \ell^2(L_\alpha)$. For every $n \in \enne$, let
  $\psi^{(n)} = (\psi^{(n)}_k)$ and
  $\bar{\psi}^{(n)} = (\bar{\psi}^{(n)}_k)$ be defined as
  \begin{align*}
    \psi_k^{(n)}(x) &:= \chi_n(x) \int_x^\infty \psi_k(y)\,dy,\\
    \bar{\psi}_k^{(n)}(x) &:= \chi'_n(x) \int_x^\infty \psi_k(y)\,dy.
  \end{align*}
  Then $\psi^{(n)}$ and $\bar{\psi}^{(n)}$ belong to
  $\ell^2(L^2_\alpha)$ for every $n \in \enne$, and
  \[
  \lim_{n \to \infty} \norm[\big]{\bar{\psi}^{(n)}}_{\ell^2(L^2_\alpha)} = 0.
  \]
\end{lemma}
\begin{proof}
  Thanks to the estimate
  \[
  \int_x^\infty \psi_k(y)\,dy = \int_x^\infty \psi_k(y)
  e^{\frac{\alpha}{2}y} e^{-\frac{\alpha}{2}y} \,dy \lesssim_\alpha
  \biggl( \int_x^\infty \abs{\psi_k(y)}^2 e^{\alpha y}\,dy
  \biggr)^{1/2} e^{-\frac{\alpha}{2}x},
  \]
  one has
  \begin{align*}
  \norm[\big]{\psi_k^{(n)}}^2_{L^2_\alpha}
  &= \int_0^\infty \chi_n^2(x) \abs[\bigg]{\int_x^\infty \psi_k(y)\,dy}^2
  e^{\alpha x}\,dx\\
  &\lesssim_\alpha \int_0^{n+1} \norm[\big]{\psi_k}^2_{L^2_\alpha}\,dx
  = (n+1) \norm[\big]{\psi_k}^2_{L^2_\alpha},
  \end{align*}
  which implies
  \[
  \norm[\big]{\psi^{(n)}}_{\ell^2(L^2_\alpha)} \lesssim_\alpha \sqrt{n+1} 
  \norm[\big]{\psi}_{\ell^2(L^2_\alpha)} < \infty.
  \]
  Similarly, one has
  \begin{align*}
  \norm[\big]{\bar{\psi}_k^{(n)}}^2_{L^2_\alpha}
  &= \int_0^\infty \chi'_n(x)^2 \abs[\bigg]{\int_x^\infty
    \psi_k(y)\,dy}^2 e^{\alpha x}\,dx\\
  &\lesssim_\alpha \int_n^{n+1}
  \norm[\big]{\ind{[n,\infty[}\psi_k}^2_{L^2_\alpha} \,dx \leq
  \norm[\big]{\ind{[n,\infty[}\psi_k}^2_{L^2_\alpha},
  \end{align*}
  hence also
  \[
  \norm[\big]{\bar{\psi}^{(n)}}_{\ell^2(L^2_\alpha)} \lesssim_\alpha
  \norm[\big]{\ind{[n,\infty[} \psi}_{\ell^2(L^2_\alpha)} 
  \leq \norm[\big]{\psi}_{\ell^2(L^2_\alpha)} < \infty.
  \]
  Moreover,
  \[
  \lim_{n \to \infty} \norm[\big]{\ind{[n,\infty[} \psi}_{\ell^2(L^2_\alpha)} = 0
  \]
  by the dominated convergence theorem.
\end{proof}

\begin{lemma}
  \label{lm:52}
  Assume that Hypothesis~1 is satisfied. Then, for every $n \in \enne$,
  \begin{itemize}
  \item[\emph{(a)}] $\abs[\big]{\sigma_k^{(n)}(\omega,t,x,r)} \leq
    \psi_k^{(n)}(x) \, \eta_k(r)$ for all $(\omega,t,x,r) \in \Omega
    \times \erre_+^2 \times \erre$;
  \item[\emph{(b)}] $\abs[\big]{\sigma_k^{(n)}(\omega,t,x,r_1) -
      \sigma_k^{(n)}(\omega,t,x,r_2)} \leq \psi_k^{(n)}(x) \,
    \bigl(\eta_k(r_1) + \eta_k(r_2)\bigr) \, \abs[\big]{r_1-r_2}$ for
    all $(\omega,t,x) \in \Omega \times \erre_+^2$ and $r_1, r_2 \in
    \erre$.
  \end{itemize}
\end{lemma}
\begin{proof}
  Since $\sigma(\cdot,r)$ is zero at infinity for all $r \in \erre$,
  the fundamental theorem of calculus yields
  \[
  \abs[\big]{\sigma_k(x,r)} \leq \int_x^\infty \abs[\big]{%
    \partial_1 \sigma_k(y,r)}\,dy \leq \eta_k(r) \int_x^\infty
  \psi_k(y)\,dy.
  \]
  Therefore, by definition of $\psi_k^{(n)}$, it follows that
  \[
  \abs[\big]{\sigma_k^{(n)}(x,r)} \leq \psi_k^{(n)}(x)\,\eta_k(r).
  \]
  The proof of (i) is thus complete, and the proof of (ii) is entirely
  similar, hence omitted.
\end{proof}

\begin{lemma}
  \label{lm:53}
  Assume that Hypothesis~1 is satisfied. Then, for every $n \in \enne$,
  \begin{itemize}
  \item[\emph{(a)}] $\lim_{x\to\infty} \sigma_k^{(n)}(\omega,t,x,r)=0$ for
    all $(\omega,t,x,r) \in \Omega \times \erre_+^2 \times \erre$;
  \item[\emph{(b)}] $\abs[\big]{\partial_1 \sigma^{(n)}_k(\omega,t,x,r)} \leq
    \bigl(\psi_k(x)+\bar{\psi}_k^{(n)}(x) \bigr) \eta_k(r)$ for all
    $(\omega,t,x,r) \in \Omega \times \erre_+^2 \times \erre$;
  \item[\emph{(c)}] $\abs[\big]{\partial_2 \sigma^{(n)}_k(\omega,t,x,r)} \leq
    \eta_k(r)$ for all $(\omega,t,x,r) \in \Omega \times \erre_+^2
    \times \erre$;
  \item[\emph{(d)}] $\abs[\big]{\partial_1 \sigma^{(n)}_k(\omega,t,x,r_1)
      - \partial_1 \sigma^{(n)}_k(\omega,t,x,r_2)} \leq \psi_k(x)
    \bigl(\eta_k(r_1) + \eta_k(r_2)\bigr) \abs{r_1-r_2}$ for all
    $(\omega,t,x) \in \Omega \times \erre_+^2$ and $r_1, r_2 \in
    \erre$;
  \item[\emph{(e)}] $\abs[\big]{\partial_2 \sigma^{(n)}_k(\omega,t,x,r_1)
      - \partial_2 \sigma^{(n)}_k(\omega,t,x,r_2)} \leq \bigl(\eta_k(r_1) +
    \eta_k(r_2)\bigr) \abs{r_1-r_2}$ for all $(\omega,t,x) \in \Omega
    \times \erre_+^2$ and $r_1, r_2 \in \erre$.
  \end{itemize}
\end{lemma}
\begin{proof}
  It is enough to prove (b) and (d), as (a), (c), and (e) are trivial.
  The chain rule yields
  \begin{equation}
    \label{eq:cat}
  \partial_1 \sigma_k^{(n)}(x,r) = \partial_1 \sigma_k(x,r) \chi_n(x)
  + \sigma_k(x,r) \chi'_n(x),
  \end{equation}
  hence, by definition of $\bar{\psi}_k^{(n)}$,
  \[
  \abs[\big]{\partial_1 \sigma_k^{(n)}(x,r)} \leq \chi_n(x) \psi_k(x)
  \eta_k(r) + \bar{\psi}_k^{(n)}(x) \eta_k(r).
  \]
  The proof of (b) is thus completed (recalling that $\chi_n(x) \in
  [0,1]$ for all $x \in \erre$). By a similar computation,
  \[
  \abs[\big]{\partial_1 \sigma_k^{(n)}(x,r_1) - \partial_1
    \sigma_k^{(n)}(x,r_2)} \leq \bigl( \chi_n(x) \psi_k(x) +
  \bar{\psi}_k^{(n)}(x) \bigr) \bigl(\eta_k(r_1) + \eta_k(r_2)\bigr)
  \abs[\big]{r_1-r_2},
  \]
  so that (d) is also proved.
\end{proof}

Let $\sigma^{(n)}$ be the map formally defined on
$\Omega \times \erre_+ \times H_\alpha$ as
\[
  \sigma^{(n)}(\omega,t,v) := \sum_{k=1}^\infty
  \sigma_k^{(n)}(\omega,t,\cdot,v(\cdot))\,e_k.
\]
If hypothesis~1 is satisfied, Then $\sigma^{(n)}$ is a well-defined
map with values in $\cL^2(U,H_\alpha)$, and it is measurable with
respect to the $\sigma$-algebra
$\mathscr{P} \otimes \mathscr{B}(H_\alpha)$. In fact, by
lemma~\ref{lm:53} $(\sigma^{(n)}_k)_{k \in \enne}$ satisfies
hypothesis~1 with $(\psi_k)$ replaced by
$(\psi_k + \bar{\psi}^{(n)}_k)$, where, by lemma~\ref{lm:pepsi}, the
$\ell^2(L^2_\alpha)$ norm of $(\bar{\psi}^{(n)}_k)$ is dominated by
the norm of $(\psi_k)$. The claim hence follows by (the proof of)
proposition~\ref{prop:pippo}. The same reasoning shows that
\[
  \beta^{(n)} := \ip[\big]{\sigma^{(n)}}{I\sigma^{(n)}}
\]
is a well-defined map from $\Omega \times \erre_+ \times H_\alpha$ to
$H_\alpha$, satisfying the same measurability properties of $\beta$.
\begin{lemma}
  \label{lm:54}
  Let Hypothesis~1 be satisfied. Then
  \begin{gather*}
  \lim_{n \to \infty} \norm[\big]{\sigma^{(n)}(\omega,t,v) -
    \sigma(\omega,t,v)}_{\cL^2(U,H_\alpha)} = 0,\\
  \lim_{n \to \infty} \norm[\big]{\beta^{(n)}(\omega,t,v) -
    \beta(\omega,t,v)}_{H_\alpha} = 0
  \end{gather*}
  for all $(\omega,t) \in \Omega \times \erre_+$ and $v \in H_\alpha$.
\end{lemma}
\begin{proof}
  For any $v \in H_\alpha$, one has (omitting the arguments $\omega$
  and $t$ throughout for simplicity)
  \begin{align*}
    \norm[\big]{\sigma^{(n)}(v) - \sigma(v)}_{\cL^2(U,H_\alpha)}
    &\leq \Bigl( \sum_{k=1}^\infty \norm[\big]{%
    \partial_1 \sigma_k^{(n)}(\cdot,v) 
    - \partial_1 \sigma_k(\cdot,v)}^2_{L^2_\alpha} \Bigr)^{1/2}\\
    &\quad + \Bigl( \sum_{k=1}^\infty \norm[\big]{\bigl(%
    \partial_2 \sigma_k^{(n)}(\cdot,v) 
    - \partial_2 \sigma_k(\cdot,v)\bigr)v'}^2_{L^2_\alpha} \Bigr)^{1/2}.
  \end{align*}
  It follows by \eqref{eq:cat} and by the triangle inequality in
  $\ell^2(L^2_\alpha)$ that
  \begin{align*}
    &\Bigl( \sum_{k=1}^\infty \norm[\big]{%
    \partial_1 \sigma_k^{(n)}(\cdot,v) 
    - \partial_1 \sigma_k(\cdot,v)}^2_{L^2_\alpha} \Bigr)^{1/2}\\
    &\hspace{3em} \leq \biggl( \sum_{k=1}^\infty \int_0^\infty
    \abs[\big]{\partial_1 \sigma_k(x,v(x))%
    (1-\chi_n(x))}^2 e^{\alpha x}\,dx \biggr)^{1/2}\\
    &\hspace{3em} \quad + \biggl(\sum_{k=1}^\infty \int_0^\infty
    \abs[\big]{\sigma_k(x,v(x))%
    \chi'_n(x)}^2 e^{\alpha x}\,dx \biggr)^{1/2},
  \end{align*}
  where the first term on the right-hand side converges to zero by
  Hypothesis~1(b) and the dominated convergence theorem. Moreover, an
  argument entirely analogous to the proof of Lemma~\ref{lm:52} yields
  \[
  \abs[\big]{\sigma_k(x,r) \chi'_n(x)} \leq \bar{\psi}_k^{(n)}(x)
  \eta_k(r) \qquad \forall (x,r) \in \erre_+ \times \erre,
  \]
  hence the second term on the right-hand side of the previous
  inequality is dominated by
  \[
  \Bigl( \sum_k^\infty \norm[\big]{\psi^{(n)}_k}^2_{L^2_\alpha} 
  \norm[\big]{\eta_k(v)}^2_{L^\infty} \Bigr)^{1/2}
  \lesssim_\alpha \norm[\big]{\psi^{(n)}}_{\ell^2(L^2_\alpha)} 
  \widetilde{\eta}\bigl( \delta \norm{v}_{H_\alpha} \bigr),
  \]
  which converges to zero by Lemma~\ref{lm:pepsi}.
  Furthermore, it follows by the identity
  \[
  \bigl( \partial_2 \sigma_k^{(n)}(\cdot,v) 
  - \partial_2 \sigma_k(\cdot,v)\bigr)v'
  = (1-\chi_n) \partial_2 \sigma_k(\cdot,v) v',
  \]
  Lemma \ref{lm:53}(c), and the dominated convergence theorem, that
  \[
  \lim_{n \to \infty} \sum_{k=1}^\infty \norm[\big]{\bigl(%
    \partial_2 \sigma_k^{(n)}(\cdot,v) 
    - \partial_2 \sigma_k(\cdot,v)\bigr)v'}^2_{L^2_\alpha} = 0,
  \]
  thus proving the pointwise convergence of $\sigma^{(n)}$ to
  $\sigma$. The pointwise convergence of $\beta^{(n)}$ to $\beta$ is
  an immediate consequence of lemma~\ref{lm:hagi}, which yields
  \[
    \norm[\big]{\beta^{(n)}(v)-\beta(v)}_{H_\alpha} \lesssim \bigl(
    \norm[\big]{\sigma^{(n)}(v)}_{\cL^2} + \norm[\big]{\sigma(v)}_{\cL^2}
    \bigr) \norm[\big]{\sigma^{(n)}(v) - \sigma(v)}_{\cL^2}.
    \qedhere
  \]
\end{proof}

We are now going to show that the sequence of maps $(\sigma^{(n)})$ is
uniformly locally Lipschitz continuous, hence, as a consequence, that
the same holds for the sequence $(\beta^{(n)})$.
\begin{lemma}
  \label{lm:55}
  For every $R \geq 0$ there exists a constant $N$, independent of
  $n$, such that
  \begin{align*}
    \norm[\big]{\sigma^{(n)}(\omega,t,v_1) -
    \sigma^{(n)}(\omega,t,v_2)}_{\cL^2(U,H_\alpha)}
    &\leq N \norm[\big]{v_1-v_2}_{H_\alpha},\\
    \norm[\big]{\beta^{(n)}(\omega,t,v_1) -
    \beta^{(n)}(\omega,t,v_2)}_{\cL^2(U,H_\alpha)}
    &\leq N \norm[\big]{v_1-v_2}_{H_\alpha}
  \end{align*}
  for all $(\omega,t) \in \Omega \times \erre_+$ and
  $v_1, v_2 \in H_\alpha$ with $\norm{v_1}_{H_\alpha}$,
  $\norm{v_2}_{H_\alpha} \leq R$.
\end{lemma}
\begin{proof}
  Lemma \ref{lm:53} implies that hypothesis~1 holds, for every
  $k \in \enne$, with $\sigma_k$ replaced by $\sigma_k^{(n)}$ and
  $\psi_k$ replaced by $\psi_k+\bar{\psi}_k^{(n)}$. Moreover, the
  $\ell^2(L^2_\alpha)$ norm of $\bar{\psi}_k^{(n)}$ is bounded by the
  norm of $\psi_k$ thanks to lemma~\ref{lm:pepsi}. Then
  proposition~\ref{prop:pippo} implies
  \begin{align*}
    \norm[\big]{\sigma^{(n)}(v_1) - \sigma^{(n)}(v_2)}_{\cL^2(U,H_\alpha)}
    &\lesssim_\alpha \Bigl( \widetilde{\eta}\bigl(%
    \delta \norm{v_1}_{H_\alpha} \bigr) + \widetilde{\eta}\bigl(%
      \delta \norm{v_2}_{H_\alpha} \bigr) \Bigr) \, \cdot\\
    &\qquad \cdot \, \bigl(%
    \norm{v_2}_{H_\alpha} + \norm{\psi}_{\ell^2(L^2_\alpha)} \bigr)
    \norm[\big]{v_1-v_2}_{H_\alpha},
  \end{align*}
  which proves the uniform local Lipschitz continuity of
  $\sigma^{(n)}$. Again by lemmata~\ref{lm:53} and \ref{lm:pepsi} and
  (the proof of) proposition~\ref{prop:pippo} it follows that
  \[
  \norm[\big]{\sigma^{(n)}(v)}_{\cL^2(U,H_\alpha)}
    \lesssim_\alpha \widetilde{\eta}\bigl(%
    \delta \norm{v}_{H_\alpha} \bigr) \,
    \bigl( \norm{v}_{H_\alpha} + \norm{\psi}_{\ell^2(L^2_\alpha)} \bigr)
  \]
  for all $v \in H_\alpha$. For any $v_1, v_2 \in H_\alpha$ with norm
  bounded by $R$, Lemma~\ref{lm:hagi} then implies
  \begin{align*}
    \norm[\big]{\beta^{(n)}(v_1) - \beta^{(n)}(v_2)}
    &\lesssim \bigl( \norm[\big]{\sigma^{(n)}(v_1)}_{\cL^2}
      + \norm[\big]{\sigma^{(n)}(v_2)}_{\cL^2} \bigr)
      \norm[\big]{\sigma^{(n)}(v_1) - \sigma^{(n)}(v_2)}_{\cL^2}\\
    &\lesssim \widetilde{\eta}(\delta R) \, \bigl( R
      + \norm{\psi}_{\ell^2(L^2_\alpha)} \bigr) \norm{v_1-v_2}_{H_\alpha},
  \end{align*}
  where the implicit constant depends on $\alpha$ and $R$, but not on
  $n$.
\end{proof}

\section{Approximations II}
\label{sec:approx2}
For every $n \in \enne$, let $\phi_n \in C^\infty_b(\erre)$ be a
primitive of the smooth cut-off function $\chi_n$, such that $\phi_n$
is odd and $\phi_n(0)=0$. Then $\phi_n$ coincides with the identity
function on $[-n,n]$ and is bounded from below and from above by
$-(n+1)$ and $n+1$, respectively.

For every $k$ and $n \in \enne$, let us define the functions
$\sigma^n_k: \Omega \times \erre_+^2 \times \erre \to \erre$ as
\[
  \sigma^n_k(\omega,t,x,r) :=   \sigma_k(\omega,t,x,\phi_n(r)).
\]
We are going to show that, for each $n \in \enne$, $(\sigma_k^n)$
satisfies hypothesis~1 with $\eta_k$ replaced by
$3\eta_k \circ \phi_n$.
\begin{lemma}
  \label{lm:kn}
  Assume that hypothesis~1 is satisfied. For every $k$ and $n \in \enne$ one
  has
  \begin{itemize}
  \item[\emph{(a)}] $\lim_{x\to\infty} \sigma_k^n(\omega,t,x,r)=0$ for
    all $(\omega,t,x,r) \in \Omega \times \erre_+^2 \times \erre$;
  \item[\emph{(b)}] $\abs[\big]{\partial_1 \sigma^n_k(\omega,t,x,r)} \leq
    \psi_k(x) \eta_k(\phi_n(r))$ for all
    $(\omega,t,x,r) \in \Omega \times \erre_+^2 \times \erre$;
  \item[\emph{(c)}] $\abs[\big]{\partial_2 \sigma^n_k(\omega,t,x,r)} \leq
    \eta_k(\phi_n(r))$ for all $(\omega,t,x,r) \in \Omega \times \erre_+^2
    \times \erre$;
  \item[\emph{(d)}]
    $\abs[\big]{\partial_1 \sigma^n_k(\omega,t,x,r_1) - \partial_1
      \sigma^n_k(\omega,t,x,r_2)} \leq \psi_k(x)
    \bigl(\eta_k(\phi_n(r_1)) + \eta_k(\phi_n(r_2))\bigr)
    \abs{r_1-r_2}$ for all $(\omega,t,x) \in \Omega \times \erre_+^2$
    and $r_1, r_2 \in \erre$;
  \item[\emph{(e)}] $\abs[\big]{\partial_2 \sigma^n_k(\omega,t,x,r_1)
      - \partial_2 \sigma^n_k(\omega,t,x,r_2)} \leq \bigl(\eta_k(\phi_n(r_1)) +
    3\eta_k(\phi_n(r_2))\bigr) \abs{r_1-r_2}$ for all $(\omega,t,x) \in \Omega
    \times \erre_+^2$ and $r_1, r_2 \in \erre$.
  \end{itemize}
\end{lemma}
\begin{proof}
  Claim (a) is trivial. The remaining ones are based on the identities
  \[
    \partial_1 \sigma_k^n(x,r) = \partial_1 \sigma_k(x,\phi_n(r)),
    \qquad \partial_2 \sigma_k^n(x,r) = \partial_2
    \sigma_k(x,\phi_n(r)) \phi'_n(r),
  \]
  where, as usual, we omit the arguments $\omega$ and $t$.
  In particular, (b) follows immediately, as well as (c), recalling
  that $\abs{\phi'_n} \leq 1$. Similarly, one has
  \begin{align*}
    &\abs[\big]{\partial_1 \sigma^n_k(x,r_1) - \partial_1 \sigma^n_k(x,r_2)}\\
    &\hspace{3em} \leq \psi_k(x)
    \bigl(\eta_k(\phi_n(r_1)) + \eta_k(\phi_n(r_2))\bigr)
    \abs[\big]{\phi_n(r_1) - \phi_n(r_2)},
  \end{align*}
  where $\abs{\phi_n(r_1) - \phi_n(r_2)} \leq \abs{r_1-r_2}$ because
  the Lipschitz constant of $\phi_n$ is one, hence (d) is verified. It
  remains to show that (e) holds true: one has
  \begin{align*}
    \abs[\big]{\partial_2 \sigma^n_k(x,r_1)
    - \partial_2 \sigma^n_k(x,r_2)}
    &= \abs[\big]{\partial_2 \sigma_k(x,r_1)\phi'_n(r_1)
      - \partial_2 \sigma_k(x,r_2)\phi'_n(r_2)}\\
    &\leq \abs[\big]{\partial_2 \sigma_k(x,r_1)\phi'_n(r_1)
      - \partial_2 \sigma_k(x,r_2)\phi'_n(r_1)}\\
    &\quad + \abs[\big]{\partial_2 \sigma_k(x,r_2)\phi'_n(r_1)
      - \partial_2 \sigma_k(x,r_2)\phi'_n(r_2)},
  \end{align*}
  where, recalling that $\abs{\phi'_n} \leq 1$,
  \begin{align*}
    \abs[\big]{\partial_2 \sigma_k(x,r_1)\phi'_n(r_1) - \partial_2
    \sigma_k(x,r_2)\phi'_n(r_1)}
    &\leq \bigl( \eta_k(\phi_n(r_1)) +
      \eta_k(\phi_n(r_2)) \bigr) \abs[\big]{\phi_n(r_1)-\phi_n(r_2)}\\
    &\leq \bigl( \eta_k(\phi_n(r_1)) +
      \eta_k(\phi_n(r_2)) \bigr) \abs{r_1 - r_2},
  \end{align*}
  and, thanks to the estimate $\abs{\phi_n''} \leq 2$,
  \begin{align*}
    \abs[\big]{\partial_2 \sigma_k(x,r_2)\phi'_n(r_1)
    - \partial_2 \sigma_k(x,r_2)\phi'_n(r_2)}
    &\leq \eta_k\bigl(\phi_n(r_2)\bigr) \abs[\big]{\phi'_n(r_1)-\phi'_n(r_2)}\\
    &\leq 2\eta_k\bigl(\phi_n(r_2)\bigr) \abs{r_1 - r_2},
  \end{align*}
  thus establishing (e) as well.
\end{proof}

The definition of $\sigma^n$ as a map from $\Omega \times \erre_+
\times H_\alpha \to \cL^2(U,H_\alpha)$, as well as the corresponding
measurability properties, follows, \emph{mutatis mutandis}, as for
$\sigma^{(n)}$ in the previous section.

\begin{lemma}
  \label{lm:62}
  Let hypothesis~1 be satisfied. Then
  \begin{gather*}
  \lim_{n \to \infty} \norm[\big]{\sigma^n(\omega,t,v) -
    \sigma(\omega,t,v)}_{\cL^2(U,H_\alpha)} = 0,\\
  \lim_{n \to \infty} \norm[\big]{\beta^n(\omega,t,v) -
    \beta(\omega,t,v)}_{H_\alpha} = 0
  \end{gather*}
  for all $(\omega,t) \in \Omega \times \erre_+$ and $v \in H_\alpha$.
\end{lemma}
\begin{proof}
  By definition of the functions $\sigma_k^n$ it follows that, for any
  $v \in H_\alpha$,
  \begin{align*}
    \norm[\big]{\sigma^n(v)-\sigma(v)}_{\cL^2(U,H_\alpha)}
    &\leq \Bigl( \sum_{k=1}^\infty \norm[\big]{%
      \partial_1\sigma_k(\cdot,\phi_n(v))
      - \partial_1\sigma_k(\cdot,v)}^2_{L^2_\alpha} \Bigr)^{1/2}\\
    &\quad + \Bigl( \sum_{k=1}^\infty \norm[\big]{%
      \partial_2\sigma_k(\cdot,\phi_n(v))\phi'_n(v)v'
      - \partial_2\sigma_k(\cdot,v)v'}^2_{L^2_\alpha} \Bigr)^{1/2}.
  \end{align*}
  Recalling that, as $n \to \infty$, $\phi_n$ converges pointwise to
  the identity function and $\phi'_n$ converges pointwise from below
  to the function identically equal to one, the claim follows by parts
  (b) and (c) of Lemma~\ref{lm:kn}, the obvious estimate $\eta_k \circ
  \phi_n \leq \eta_k$, and the dominated convergence theorem. The
  pointwise convergence of $\beta^n$ follows exactly as in proof of
  lemma~\ref{lm:54}.
\end{proof}

\begin{lemma}
  \label{lm:63}
  Let hypothesis~1 be satisfied. For every $R \geq 0$ there exists a
  constant $N$, independent of $n$, such that
  \begin{align*}
    \norm[\big]{\sigma^n(\omega,t,v_1) - \sigma^n(\omega,t,v_2)}_{\cL^2(U,H_\alpha)}
    &\leq N \norm[\big]{v_1-v_2}_{H_\alpha}\\
    \norm[\big]{\beta^n(\omega,t,v_1) - \beta^n(\omega,t,v_2)}_{H_\alpha}
    &\leq N \norm[\big]{v_1-v_2}_{H_\alpha}
  \end{align*}
  for all $(\omega,t) \in \Omega \times \erre_+$ and
  $v_1, v_2 \in H_\alpha$ with $\norm{v_1}_{H_\alpha}$,
  $\norm{v_2}_{H_\alpha} \leq R$.
\end{lemma}
\begin{proof}
  Lemma \ref{lm:kn} implies that hypothesis~1 holds with $\sigma_k$
  replaced by $\sigma_k^n$ and $\eta_k$ replaced by
  $3 \eta_k \circ \phi_n$. Since $\abs{\phi_n(r)} \leq \abs{r}$ for
  all $r \in \erre$ and $\eta_k$ is even and increasing, one has
  $\eta_k \circ \phi_n \leq \eta_k$, hence
  Proposition~\ref{prop:pippo} yields
  \begin{align*}
    \norm[\big]{\sigma^n(\omega,t,v_1)-\sigma^n(\omega,t,v_2)}_{\cL^2(U,H_\alpha)}
    &\lesssim_\alpha
    \Bigl( \widetilde{\eta}\bigl(%
    \delta \norm{v_1}_{H_\alpha} \bigr) + \widetilde{\eta}\bigl(%
      \delta \norm{v_2}_{H_\alpha} \bigr) \Bigr) \, \cdot\\
    &\qquad \cdot \, \bigl(%
    \norm{v_2}_{H_\alpha} + \norm{\psi}_{\ell^2(L^2_\alpha)} \bigr)
    \norm[\big]{v_1-v_2}_{H_\alpha}.
  \end{align*}
  Denoting the implicit constant in this inequality by $c(\alpha)$,
  setting
  \[
    N := 2 c(\alpha) \, \widetilde{\eta}(\delta R) \bigl( R +
    \norm{\psi}_{\ell^2(L^2_\alpha)} \bigr),
  \]
  the claim regarding $\sigma^n$ follows.
  Moreover, again by the inequality $\eta_k \circ \phi_n \leq \eta_k$,
  lemma~\ref{lm:kn} and the (proof of) proposition~\ref{prop:pippo}
  imply
  \[
  \norm[\big]{\sigma^{(n)}(\omega,t,v)}_{\cL^2(U,H_\alpha)}
    \lesssim_\alpha \widetilde{\eta}\bigl(%
    \delta \norm{v}_{H_\alpha} \bigr) \,
    \bigl( \norm{v}_{H_\alpha} + \norm{\psi}_{\ell^2(L^2_\alpha)} \bigr)
  \]
  for every $n \in \enne$, $v \in H_\alpha$, and $(\omega,t) \in
  \Omega \times \erre_+$. The claim about $\beta^n$ then follows by
  lemma~\ref{lm:hagi}, as in the proof of lemma~\ref{lm:55}.
\end{proof}

\section{Positivity of forward rates}
\label{sec:fw}
\begin{thm}
  Assume that hypothesis~1 is fulfilled and that
  \[
  \abs{\sigma_k(x,r)} \ind{\{r \leq 0\}} \lesssim r^-,
  \]
  or, more generally, that
  \[
  \abs{\sigma_k(x,r)} \leq \abs{r} \, \eta_k(r) \,
  \int_x^\infty \psi_k(y)\,dy
  \]
  For the latter to hold it is sufficient that $\sigma(x,0)=0$.
  Then forward rates are positive.  
\end{thm}
\begin{proof}
  For every $k \in \enne$ and $n,m \in \enne$, let
  $\sigma_k^{n,m} := \bigl(\sigma_k^{(n)}\bigr)^m$, i.e.
  \[
    \sigma_k^{n,m}(x,r) = \sigma_k(x,\phi_m(r))\chi_n(x).
  \]
  Lemmata~\ref{lm:52} and \ref{lm:kn} imply that, for each
  $n,m \in \enne$, $\sigma_k^{n,m}$ satisfies hypothesis~1 with
  $\psi_k$ and $\eta_k$ replaced by $\psi_k+\bar{\psi}_k^{(n)}$ and
  $3\eta_k \circ \phi_n$, respectively. Since, by
  lemma~\ref{lm:pepsi}, the $\ell^2(L^2_\alpha)$ norm of
  $\bar{\psi}^{(n)}$ is dominated by the one of $\psi$, and
  $\eta_k \circ \phi_n \leq \eta_k$ for every $k,n \in \enne$,
  proposition~\ref{prop:pippo} implies that the corresponding map
  $\sigma^{n,m}: \Omega \times \erre_+ \times H_\alpha \to
  \cL^2(U,H_\alpha)$ is well defined for every $n,m \in \enne$, and
  that it is locally bounded and locally Lipschitz continuous in its
  third argument, uniformly with respect to the other ones. In
  particular, setting
  \[
    \beta^{n,m} := \ip[\big]{\sigma^{n,m}}{I\sigma^{n,m}},
  \]
  the equation
  \begin{equation}
    \label{eq:mona}
    dv + Av\,dt = \beta^{n,m}(v)\,dt + \sigma^{n,m}(v)\,dW,
    \qquad u^{n,m}(0)=u_0,
  \end{equation}
  admits a unique $H_\alpha$-valued mild solution $u^{n,m}$ defined on a
  maximal stochastic interval $\co{0}{T_{n,m}}$.

  Let us show that, for every $n,m \in \enne$, $(\sigma_k^{n,m})$
  satisfies the assumptions of proposition~\ref{prop:oli}:
  lemmata~\ref{lm:52} and \ref{lm:kn} imply that
  \[
    \abs[\big]{\sigma^{n,m}_k(x,r)} \lesssim \psi_k^{(n)}(x) \,
    \eta_k(\phi_m(r))
  \]
  as well as, recalling that $\abs{\phi_m'} \leq 1$,
  \begin{align*}
    &\abs[\big]{\sigma^{n,m}_k(x,r_1) - \sigma^{n,m}_k(x,r_2)}\\
    &\hspace{3em}\lesssim \psi_k^{(n)}(x) \, \bigl( \eta_k(\phi_m(r_1))
      + \eta_k(\phi_m(r_2)) \bigr) \, \abs[\big]{\phi_m(r_1)-\phi_m(r_2)}\\
    &\hspace{3em}\leq \psi_k^{(n)}(x) \, \bigl( \eta_k(\phi_m(r_1))
      + \eta_k(\phi_m(r_2)) \bigr) \, \abs{r_1-r_2}.
  \end{align*}
  Since $\abs{\phi_m} \leq m+1$, it follows that
  \[
    \abs[\big]{\sigma_k^{n,m}(x,r)} \lesssim \psi_k^{(n)}(x) \,
    \eta_k(m+1),
  \]
  where
  $\norm[\big]{\psi_k^{(n)} \, \eta_k(m+1)}_{L^2_\alpha} = \eta_k(m+1)
  \norm[\big]{\psi_k^{(n)}}_{L^2_\alpha}$, hence, by the
  Cauchy-Schwarz inequality and the estimate
  $\norm{\cdot}_{\ell^4} \leq \norm{\cdot}_{\ell^2}$,
  \begin{align*}
    \norm[\big]{\bigl( \psi_k^{(n)} \, \eta_k(m+1) \bigr)}_{\ell^2(L^2_\alpha)}
    &= \Bigl( \sum_{k=1}^\infty \eta_k^2(m+1)
      \norm[\big]{\psi_k^{(n)}}^2_{L^2_\alpha} \Bigr)^{1/2}\\
    &\leq \Bigl( \sum_{k=1}^\infty \eta_k^4(m+1) \Bigr)^{1/4}
      \Bigl( \sum_{k=1}^\infty \norm[\big]{\psi_k^{(n)}}^4_{L^2_\alpha}
      \Bigr)^{1/4}\\
    &\leq \widetilde{\eta}(m+1) \,
      \norm[\big]{\psi^{(n)}}_{\ell^2(L^2_\alpha)}.
  \end{align*}
  Recalling that the $\ell^2(L^2_\alpha)$ norm of $\psi^{(n)}$ is
  finite by lemma~\ref{lm:pepsi}, this shows that assumption (a) of
  proposition~\ref{prop:oli} is fulfilled. Similarly, one has
  \[
    \abs[\big]{\sigma_k^{n,m}(x,r_1) - \sigma_k^{n,m}(x,r_2)}
    \lesssim \psi_k^{(n)}(x) \, \eta_k(m+1) \, \abs{r_1-r_2},
  \]
  where
  $\norm[\big]{\psi_k^{(n)} \, \eta_k(m+1)}_{L^\infty} = \eta_k(m+1)
  \, \norm[\big]{\psi_k^{(n)}}_{L^\infty}$ and
  \[
    \norm[\big]{\psi_k^{(n)}}_{L^\infty} \leq
    \norm[\big]{\psi_k}_{L^1} \lesssim
    \norm[\big]{\psi_k}_{L^2_\alpha}.
  \]
  Since $\psi \in \ell^2(L^2_\alpha)$, assumption (b) of
  proposition~\ref{prop:oli} is also satisfied. Therefore
  \eqref{eq:mona} admits a unique $L^2_{-\alpha}$-valued mild solution
  on any finite time interval. Since $H_\alpha \embed L^2_{-\alpha}$,
  it follows that the $H_\alpha$-valued solution $u^{n,m}$ coincides
  on $\co{0}{T_{n,m}}$ with the global $L^2_{-\alpha}$-valued solution,
  which we shall also denote by $u^{n,m}$.

  One has
  \[
    \abs[\big]{\sigma_k^{n,m}(x,r)} =
    \abs[\big]{\sigma_k(x,\phi_m(r)) \chi_n(x)}
    \leq \abs[\big]{\sigma_k(x,\phi_m(r))}
  \]
  and $r \leq 0$ if and only if $\phi_m(r) \leq 0$, hence
  \[
    \abs[\big]{\sigma_k^{n,m}(x,r)} \ind{\{r \leq 0\}}
    \leq \abs[\big]{\sigma_k(x,\phi_m(r))}
    \ind{\{\phi_m(r) \leq 0\}} \lesssim \phi_m(r)^- \leq r^-.
  \]
  Under the more general hypothesis on $(\sigma_k)$, one has
  \[
    \abs[\big]{\sigma_k^{n,m}(x,r)} \ind{\{r \leq 0\}}
    \leq \psi_k^{(n)}(x) \, \eta_k(\phi_m(r)) \,
    \abs{r} \ind{\{r \leq 0\}}
    \leq \psi_k^{(n)}(x) \, \eta_k(m+1) \, r^-.
  \]
  Then, for any $v \in L^2_{-\alpha}$,
  \[
    \norm[\big]{\sigma_k^{n,m}(v) \ind{\{v \leq 0\}}}_{L^2_{-\alpha}}
    \leq \eta_k(m+1) \, \norm[\big]{\psi_k^{(n)}}_{L^\infty} \,
    \norm[\big]{v^-}_{L^2_{-\alpha}},
  \]
  where $\norm[\big]{\psi_k^{(n)}}_{L^\infty} \leq
  \norm[\big]{\psi_k^{(n)}}_{L^1} \lesssim
  \norm[\big]{\psi_k^{(n)}}_{L^2_\alpha}$, hence
  \[
    \norm[\big]{\sigma^{n,m}(v) \ind{\{v \leq 0\}}}_{\ell^2(L^2_{-\alpha})}
    \lesssim \widetilde{\eta}(m+1) \, 
    \norm[\big]{\psi^{(n)}}_{\ell^2(L^2_\alpha)} \, \norm[\big]{v^-}_{L^2_{-\alpha}}.
  \]
  Let us now consider
  \[
  \ip[\big]{v^-}{\beta^{n,m}(v)}_{L^2_{-\alpha}} = \sum_{k=1}^\infty
  \int_0^\infty e^{-\alpha x} v^-(x) \sigma_k^{n,m}(x,v(x)) \int_0^x
  \sigma_k^{n,m}(y,v(y))\,dy\,dx,
  \]
  where
  \[
  \int_0^x \sigma_k^{n,m}(y,v(y))\,dy \leq
  \norm[\big]{\sigma_k^{n,m}(v)}_{L^1} \lesssim
  \norm[\big]{\sigma_k^{n,m}(v)}_{L^2_\alpha}.
  \]
  We recall that, as in the proof of proposition~\ref{prop:oli}, one
  has
  \[
  \norm[\big]{\sigma_k^{n,m}(v)}_{L^2_\alpha} \leq
  \norm[\big]{\psi_k^{(n)}}_{L^2_\alpha}.
  \]
  Note that $r^- = \abs{r} \ind{\{r \leq 0\}}$ implies
  \[
  r^- \sigma_k^{n,m}(x,r) \leq \psi_k^{(n)}(x) \eta_k(m+1) \abs{r^-}^2,
  \]
  from which it follows, by the H\"older inequality,
  \begin{align*}
    \ip[\big]{v^-}{\beta^{n,m}(v)}_{L^2_{-\alpha}} &\lesssim
    \sum_{k=1}^\infty \eta_k(m+1) \,
    \norm[\big]{\sigma_k^{n,m}(v)}_{L^2_\alpha} \, \int_0^\infty
    \psi_k^{(n)}(x) \abs{v^-(x)}^2 e^{-\alpha x}\,dx\\
    &\lesssim \norm[\big]{v^-}_{L^2_{-\alpha}} \sum_{k=1}^\infty
    \eta_k(m+1) \, \norm[\big]{\sigma_k^{n,m}(v)}_{L^2_\alpha} \,
    \norm[\big]{\psi_k^{(n)}}_{L^2_\alpha}\\
    &\leq \norm[\big]{v^-}_{L^2_{-\alpha}} \sum_{k=1}^\infty
    \eta_k(m+1) \, \norm[\big]{\psi_k^{(n)}}^2_{L^2_\alpha}\\
    &\leq \norm[\big]{v^-}_{L^2_{-\alpha}} \, \widetilde{\eta}(m+1) \,
    \norm[\big]{\psi^{(n)}}^2_{\ell^2(L^2_\alpha)}.
  \end{align*}
  Theorem~\ref{thm:pos} applied to equation~\eqref{eq:mona} on the
  space $L^2_{-\alpha}$ thus yields $u^{n,m}(t) \geq 0$ for every $t
  \in \erre_+$.

  As already observed, thanks to lemma~\ref{lm:53},
  $\sigma^{(n)}$ satisfies hypothesis~1 with $\psi_k$ replaced by
  $\psi_k+\bar{\psi}^{(n)}$. Then lemmata~\ref{lm:62} and \ref{lm:63}
  imply, in view of theorem~\ref{thm:cozza}, that
  \[
    u^{n,m}\ind{\co{0}{T_n \wedge T_{n,m}}} \, \longrightarrow \,
    u^n\ind{\co{0}{T_n}}
  \]
  in $L^0(\Omega \times \erre_+;H_\alpha)$ as $m \to \infty$, where
  $u_n$ is the unique local mild solution on the maximal stochastic
  interval $\co{0}{T_n}$ to the equation
  \[
    dv + Av\,dt = \beta^n(v)\,dt + \sigma^n(v)\,dW, \qquad u^n(0)=u_0.
  \]
  By a completely similar argument,
  \[
    u^n \ind{\co{0}{T_n \wedge T}} \, \longrightarrow \,
    u \ind{\co{0}{T}}
  \]
  in $L^0(\Omega \times \erre_+;H_\alpha)$ as $n \to \infty$, where
  $u$ is the unique local mild solution on the maximal stochastic
  interval $\co{0}{T}$ to Musiela's SPDE.  Since convergence in
  $H_\alpha$ implies convergence in $C(\erre_+)$, positivity of
  $u^{n,m}$ implies positivity of $u^n$ on $\co{0}{T_n}$ for every
  $n \in \enne$, which in turns implies positivity of $u$ on
  $\co{0}{T}$.
\end{proof}

\bibliographystyle{amsplain}
\bibliography{ref,finanza}

\end{document}